\documentclass[11pt, a4paper, twoside, reqno]{amsart}
\usepackage[utf8]{inputenc}
\usepackage[T1]{fontenc}
\usepackage{lmodern}
\usepackage{microtype}
\usepackage[frak=boondox]{mathalpha}
\usepackage{dsfont}
\usepackage{bbold}
\usepackage{upgreek}
\usepackage{mathrsfs}
\usepackage{euscript}
\usepackage{amssymb}

\usepackage[top=1in, bottom=1in, left=1in, right=1in, headheight=15pt, footskip=40pt]{geometry}
\usepackage{enumitem}
\usepackage{parskip}
\usepackage{float}
\usepackage{caption}
\usepackage{tabto}
\usepackage{color}
\usepackage{mdframed}
\usepackage{hyphenat}
\usepackage{comment}

\usepackage{amsmath, amssymb, amsthm, mathtools}
\usepackage{extarrows}
\usepackage{tikz-cd}
\usepackage[all,cmtip]{xy} 
\usepackage{graphicx}

\usepackage[numbers]{natbib}
\setcitestyle{open={},close={}}
\makeatletter
\renewcommand{\@biblabel}[1]{[#1]\hfill}
\makeatother

\usepackage[hidelinks]{hyperref}
\usepackage[nameinlink]{cleveref}
\newtheoremstyle{mystyle}
  {\topsep}    
  {\topsep}   
  {\itshape}   
  {}           
  {\bfseries}   
  {.}          
  {.5em}        
  {}            

\newtheoremstyle{spacedremark} 
  {\topsep}    
  {\topsep}     
  {\normalfont} 
  {}           
  {\bfseries}   
  {.}          
  {.5em}        
  {}               

\theoremstyle{mystyle}
\newtheorem{thm}{Theorem}[subsection]\crefname{thm}{Theorem}{Theorems}
\newtheorem{lem}[thm]{Lemma}\crefname{lem}{Lemma}{Lemmas}

\newtheorem{prop}[thm]{Proposition}\crefname{prop}{Proposition}{Propositions}
\newtheorem{cor}[thm]{Corollary}\crefname{cor}{Corollary}{Corollaries}
\newtheorem{con}[thm]{Construction}
\newtheorem{obs}[thm]{Observation}\crefname{obs}{Observation}{Observations}

\newtheorem{defn}[thm]{Definition}\crefname{defn}{Definition}{Definitions}

\newtheorem{inner-recall-star}{Theorem}
\newenvironment{recall*}[1]
  {\begin{inner-recall-star}[see \Cref{#1}]}
  {\end{inner-recall-star}}
\theoremstyle{spacedremark}
\newtheorem{rem}[thm]{Remark}\crefname{rem}{Remark}{Remarks}
\newtheorem{ex}[thm]{Example}\crefname{ex}{Example}{Examples}
\newtheorem{coex}[thm]{Counterexample}\crefname{coex}{Counterexample}{Counterexamples}

\DeclareMathOperator*{\colim}{colim}

\newcommand{\dcolim}{\varinjlim}
\newcommand{\plim}{\varprojlim}

\DeclareMathAlphabet{\duc}{U}{dutchcal}{m}{n}
\SetMathAlphabet{\duc}{bold}{U}{dutchcal}{b}{n}
\DeclareFontFamily{U}{BOONDOX-calo}{\skewchar\font=45 }
\DeclareFontShape{U}{BOONDOX-calo}{m}{n}{<-> s*[1.0] BOONDOX-r-calo}{}
\DeclareFontShape{U}{BOONDOX-calo}{b}{n}{<-> s*[1.0] BOONDOX-b-calo}{}
\DeclareMathAlphabet{\mal}{U}{BOONDOX-calo}{m}{n}
\SetMathAlphabet{\mal}{bold}{U}{BOONDOX-calo}{b}{n}

\newcommand{\esc}{\EuScript}

\setlist[enumerate]{leftmargin=*, nosep}
\begin{document}

\title[$\mathbb{P}^1$ Homotopy Theory of Schemes]{$\mathbb{P}^1$ Homotopy Theory of Schemes}

\author[Dipankar Maity]{Dipankar Maity}
\address{Department of Mathematical Sciences \\
Indian Institute of Science Education and Research (IISER) Mohali \\ Knowledge city, Sector 81, SAS Nagar, Manauli PO 140306 \\
India}
\email{ph22057@iisermohali.ac.in}
\subjclass[2020]{Primary 14F42; Secondary 18N55, 18N60}
\keywords{Motivic homotopy theory, Localizations and $\infty$-categories}

\begin{abstract}
We study the $\mathbb{P}^1$-local motivic homotopy theory under the Nisnevich topology on $Sm_S$. We show that although the homotopy category $\mathcal{H}^{\mathbb{P}^1}(-)$ admits standard schematic functorialities, it fails six-functor gluing, thereby obstructing a full six-functor formalism. On the unstable side, we prove that the unstable $\mathbb{P}^1$-motivic localization preserves Nisnevich connectivity, despite the absence of an interval structure on $\mathbb{P}^1$. Over perfect fields, we then establish the $\mathbb{P}^1$-invariance of the higher unstable homotopy group sheaves $\pi_i^{\mathbb{P}^1}$ ($i \geq 1$) using the Bloch-Ogus-Gabber theorem. On the stable side, we prove the stable $\mathbb{P}^1$-connectivity theorem over fields. As a consequence, we obtain a homotopy $t$-structure on the stable $\infty$-category of $\mathbb{P}^1$-invariant Nisnevich spectra, and deduce the strict $\mathbb{P}^1$-invariance of stable $\mathbb{P}^1$-homotopy group sheaves.
\end{abstract}
\maketitle

\tableofcontents
\section{Introduction}
For more than two decades, Motivic homotopy theory [\cite{morel19991}], introduced as the homotopical version of Voevodsky's theory of mixed motives [\cite{mazza2006lecture}], has served as a unifying framework for (co)homology theories in algebraic geometry. The requirements were simple: the (co)homologies must satisfy Nisnevich descent and $\mathbb{A}^1$ homotopy invariance. Many well-known cohomology theories found their place in it. (Quillen's) Algebraic $K$-theory (over a regular base), Motivic cohomology, hermitian $K$-theory, and Algebraic cobordism are among the most notable.

Unfortunately, not all cohomology theories are $\mathbb{A}^1$-invariant. One prominent example is the Topological Hochschild homology ($THH$ for short). In fact, the $\mathbb{A}^1$-localization of $THH$ is contractible [\cite{ELMANTO2021106640}, Theorem 2.2.1] (see loc. cit. for more results of this nature). In fact, even Quillen's algebraic $K$-theory is not $\mathbb{A}^1$-invariant over an arbitrary base (its $\mathbb{A}^1$-localization is precisely Weibel's homotopy $K$-theory). It is therefore important to have other frameworks for studying motivic homotopy theory. 

There have mostly been two candidates for the same: 
\begin{enumerate}
    \item The logarithmic motivic homotopy theory [\cite{binda2023logarithmic}],
    \item The theory of (non $\mathbb{A}^1$ invariant) Motivic spectras [\cite{MR4810064}, \cite{annala2022motivic}].
\end{enumerate}
Among these two, the second appears only in spectral form: this is the study of the category of $\mathbb{P}^1$-spectra satisfying elementary blowup excision, denoted $\mathrm{MS}_S:=Sp_{\mathbb{P}^1}^{elb} (S)$ (caution: this isn't stable in general). The first, however, comes in both unstable and stable forms: $$log\mbox{-}\mathcal{H}(S):= L_{\Box}L_{nis}L_{ads}\mathcal{P}(SmlSm_S)\text{ and }log\mbox{-}\mathcal{SH}(S):= L_{\Box}L_{nis}L_{ads}\mathcal{SH}(SmlSm_S),$$ where $SmlSm$ denotes the category of smooth log smooth schemes, $\Box$ denotes the log interval $(\mathbb{P}^1_S, \infty)$, and `$ads$' stands for admissible blowups. 

Both of these candidates contain the (corresponding) ordinary $\mathbb{A}^1$ homotopy categories (stable/unstable) as full subcategories. At the same time, they also host many other natural candidates. However, the only unstable candidate at hand, namely the logarithmic motivic homotopy category, requires lifting the cohomology theory in question to a log version, which is sometimes non-standard. 

There is a third candidate, which is somewhat off from a homotopical viewpoint but allows representation of certain cohomology theories that fail $\mathbb{A}^1$-invariance. It is the category of spaces/cohomologies that are $\mathbb{P}^1$-invariant (in other words, these are weight $0$ cohomologies that are not necessarily homotopy invariant). There could be many different ways to justify its purpose. 

One of those is the famous zero-slice construction of the category of $\mathbb{A}^1$ motivic $ S^1$-spectra. Specifically, $$s_0\mathcal{SH}^{\mathbb{A}^1}_{S^1}=\mathcal{SH}^{\mathbb{A}^1}_{S^1}/f_1\simeq L_{\mathbb{P}^1}\mathcal{SH}^{\mathbb{A}^1}_{S^1}.$$ Introduced by Voevodsky in [\cite{voevodsky2002possible}], this is a powerful tool for obtaining various spectral sequences of motivic nature. In detail, given a motivic spectrum $\mal{A}$, there is a spectral sequence with $E^2$ page given by the cohomology of $L_{\mathbb{P}^1}\mal{A}$ and converging to the cohomology of $\mal{A}$. Another consequence is that, over perfect fields, the above equivalence yields $$\mathcal{SH}^{b}_{S^1}\simeq L_{\mathbb{P}^1}\mathcal{SH}^{\mathbb{A}^1}_{S^1},$$ so that studying birational homotopy is intrinsically linked to studying $\mathbb{P}^1$ localization. 

The other reason is that the $0$-th $\mathbb{P}^1$ homology of the sheaf of K\"ahler differentials vanishes [\cite{ayoub2020P1}, Theorem 6.2 (1)], which guarantees the existence of a purely arithmetic version of the Kodaira-Spencer class map. In fact, using this machinery, he proves that the Kodaira-Spencer class map is non-vanishing.

This paper aims to study the unstable analog of $\mathbb{P}^1$ homotopy theory in greater detail. We shall call a space $\esc{X}\in \mathcal{P}(Sm_S)$ a $\mathbb{P}^1$ motivic space (or projective motivic space) if it satisfies nisnevich descent and is local with respect to the set of projection maps $\{\mathbb{P}^1_X\to X\}_{X\in Sm_S}$. By formal reasons [\cite{lurie2009higher}, Proposition 5.5.4.15], the full subcategory $\mathcal{H}^{\mathbb{P}^1}(S)$ of $\mathcal{P}(Sm_S)$ consisting of $\mathbb{P}^1$-local nisnevich-local objects is an accessible localization. In particular, $\mathcal{H}^{\mathbb{P}^1}(S)$ is again a presentable $\infty$-category. 

As usual, one would like to have a model for the localization functor mentioned above. The first obstacle is that $\mathbb{P}^1$ is not an interval object [in the sense of \cite{morel19991}, Definition 2.3.1] on the site $(Sm_S, nis)$ (see the second paragraph of \S\cref{section 2}). Therefore, it is not possible to have a `singular model' for the $\mathbb{P}^1$-localization as in [\cite{morel19991}, 2.3.2]. Instead, we use a general categorical model from [\cite{hoyois2017six}, Proposition 3.4]. As one would expect:
\begin{recall*}{model for pmot localization}
The nisnevich local model of $\mathbb{P}^1$-localization can be given by $$L_{pmot}\simeq \colim _{n\to \infty }(L_{nis}\circ  L_{\mathbb{P}^1})^n$$
\end{recall*}
The reason the $\mathbb{A}^1$ connectivity property in [\cite{morel19991}, Corollary 2.3.22] holds is that the existence of a singular model guarantees a surjection $\pi_0^{nis}\to \pi_0^{nis}L_{mot}$. It turns out that this still works even if the localization is at an object $I$ that is not necessarily an `interval object' (in the sense of Morel Voevodsky), but just has a rational section $*\to I$. Since $\mathbb{P}^1$ has a rational section, we deduce: 
\begin{recall*}{P1 connectivity theorem}
For any space $\esc{X}$, the canonical map $$\pi_0^{nis}\esc{X}\to \pi_0^{nis}L_{\mathbb{P}^1}\esc{X}=: \pi_0^{\mathbb{P}^1}\esc{X}$$ is a surjective morphism of nisnevich sheaves of sets on $Sm_S$. Consequently, $L_{pmot}$ preserves nisnevich connectivity.
\end{recall*}
In the next subsection, we discuss examples of $\mathbb{P}^1$ motivic spaces. Among these, the important ones are all affine schemes (\Cref{affine algebras are projective local}), in particular the affine line $\mathbb{A}^1$. It should be noted that most affine schemes, including $\mathbb{A}^1$, are not $\mathbb{A}^1$ local. Notably, any unramified presheaf that is $\mathbb{A}^1$-invariant is $\mathbb{P}^1$-invariant \Cref{weak unram a1 is p1}. Another unramified example worth mentioning is the cohomological Brauer group. Examples of higher simplicial homotopy type include the finite etale stack \Cref{fet stack} and all birational local spaces. 

In the next small subsection, we recall from [\cite{0bat}] that a connected Nisnevich space is birational local if and only if it is $\mathbb{P}^1$-local and $\mathbb{A}^1$-local. To further compare with birationality, we introduce the standard notion of strongly $\mathbb{P}^1$/$\mathbb{A}^1$-invariant monoids and then show that:
\begin{recall*}{strong a1 p1 groups are bir}
Let $k$ be a perfect field. A grouplike Nisnevich local monoid space $\esc{G}$ is birational local iff it is strongly $\mathbb{A}^1$ invariant and strongly $\mathbb{P}^1$ invariant.
\end{recall*}
However, we caution that although all $\mathbb{A}^1$-invariant sheaves of groups are automatically $\mathbb{P}^1$-local, a similar claim fails for strong/strict invariance. We discuss the example of $\mathbb{G}_m$, which is both $\mathbb{A}^1$-local and $\mathbb{P}^1$-local, but is only strongly $\mathbb{A}^1$-invariant, not strongly $\mathbb{P}^1$-invariant. In fact, we show that
\begin{recall*}{htype of BGm}
    The canonical forgetful map $[\mathbb{A}^1/\mathbb{G}_m]\to \mathrm{B}_{nis}\mathbb{G}_m$ is a projective motivic equivalence. 
\end{recall*}

Next, we turn to more serious questions about the functoriality of $\mathcal{H}^{\mathbb{P}^1}(-)$ and, in a major setback, find that the obvious functorial behavior of $\mathcal{H}^{\mathbb{P}^1}(-)$ lacks a localization sequence. More specifically, the $\mathbb{P}^1$-local analog of [\cite{morel19991}, Theorem 3.2.21] fails in the $\mathbb{P}^1$-local setting. Another expected property is the purity isomorphism. Unfortunately, at this stage, it is not even clear to me what the correct statement of the $\mathbb{P}^1$-local purity theorem should be. In the absence of this, the $\mathbb{P}^1$-local algebraic topology is weaker even over perfect fields.

Indeed, as would be known to a practitioner of motivic homotopy theory, without purity it is almost impossible to deduce the 'strength' of the local homotopy groups. To be precise, the $\mathbb{A}^1$ purity theorem guarantees that the higher $\mathbb{A}^1$ homotopy group sheaves are strongly/strictly $\mathbb{A}^1$-invariant, so that a connected $\mathbb{A}^1$ motivic space can be built from and decomposed into its simplicial Eilenberg-Maclane layers. However, there are ways to achieve the same result. For example, one can produce an independent proof of the unstable connectivity theorem by an appropriate use of the $\mathbb{A}^1$ singular model [see \cite{ayoub2022algebraic} for details]. Since we neither have a single model nor purity at this time, we are unable to carry out the Postnikov construction and deconstructions in $\mathbb{P}^1$ homotopy theory. What I am able to obtain is the $\mathbb{P}^1$ invariance of higher $\mathbb{P}^1$ homotopy group sheaves. The method is an adaptation of the affine Bloch-Ogus-Gabber treatment to the $\mathbb{P}^1$ setting, which I learned from [\cite{ayoub2020P1}]. The idea is first to prove:
\begin{recall*}{pi0 of grouplike is weak unram}
For any $H$-group-space $\esc{X}$,  $\pi_0^{\mathbb{P}^1}\esc{X}$ is weakly unramified.
\end{recall*}
We then show that, in general, a weakly unramified weakly $\mathbb{P}^1$-invariant presheaf $G$ is $\mathbb{P}^1$-invariant (\Cref{weakly p1 + weakly unramified => p1 invariant}). Thus, we obtain a direct $\mathbb{P}^1$-local analog of [\cite{choudhury2014connectivity}]:
\begin{recall*}{pi0 of grouplike is p1 invariant}
For projective motivic $H$-group-space $\esc{G}$,  $\pi_0^{nis}\esc{G}$ is $\mathbb{P}^1$-invariant.
\end{recall*}
 Repeated application of this to the grouplike spaces $\Omega^i\esc{X}$ then helps us achieve:
\begin{recall*}{p1 inv pf pi_i^p1}
     For any $i\geq 1$ and any space $\esc{X}$, the projective homotopy groups $\pi_i^{\mathbb{P}^1}\esc{X}$ are $\mathbb{P}^1$-invariant.
\end{recall*}
To conclude, we briefly discuss the stabilization of $\mathcal{H}^{\mathbb{P}^1}$ and prove the stable $\mathbb{P}^1$ connectivity theorem over fields. There is no new idea in this section; the material is excerpted from Morel's stable connectivity [\cite{morel2005stable}] and Ayoub's $\mathbb{P}^1$ homological connectivity [\cite{ayoub2020P1}]. These are adapted to the most general setup in another work of mine [\cite{nStbat}], and I use the relevant results of {\textit{loc. cit.}} directly to arrive at the conclusions of the last sections. The only nontrivial point is that the triviality of $\mathbb{P}^1$ stable homotopy groups can be checked on $\mathcal{F}_k$, which follows exactly as in the unstable case. 
\begin{recall*}{projective stable connectivity}
    Over any field $k$, the projective motivic localization functor $L_{pmot}^{sp}$ of $S^1$-spectras, preserves nisnevich $n$ connectivity. 
\end{recall*}
The following are straightforward consequences of the connectivity result thus obtained, which again rely on well-known facts (see [\cite{nStbat}, \S4]): 
\begin{recall*}{t structure for P1 motivic spectras}
The pair of full sub-categories \begin{flalign*}
    (\mathcal{SH}^{\mathbb{P}^1}_{S^1}\bigcap \mathcal{SH}_{S^1,\geq 0},\mathcal{SH}^{\mathbb{P}^1}_{S^1}\bigcap \mathcal{SH}_{S^1,\leq 0})
\end{flalign*} is an accessible complete $t$-structure on $\mathcal{SH}^{\mathbb{P}^1}_{S^1}$. The heart of this $t$-structure coincides with the category of strictly $\mathbb{P}^1$-invariant sheaves of abelian groups.
\end{recall*}
\begin{recall*}{LAYERS OF P1 SPECTRA}
The following conditions are equivalent for a nisnevich sheaf $\mal{A}$ of spectras on $k$
    \begin{enumerate}
        \item $\mal{A}$ is $\mathbb{P}^1$-local.
        \item The nisnevich connective covers $\mal{A}_{\geq n}$ are $\mathbb{P}^1$-local for all $n$.
        \item The nisnevich sheaves of homotopy groups $\pi_n^{sp,nis}\mal{A}$ are strictly $\mathbb{P}^1$ invariant for all $n$.
    \end{enumerate}
  \end{recall*}

\subsection{Notations and terminologies}
Throughout this paper, $S$ will denote a qcqs scheme, and $k$ will denote a field. We do not assume $k$ is perfect unless otherwise specified. By a morphism of schemes, we shall always mean a separated morphism. Smooth morphisms will be assumed to be quasicompact and hence of finite presentation.

We shall freely use the language of $\infty$-categories without referring to any particular model for the theory. The reader may freely use their preferred model. A standard reference is [\cite{lurie2009higher}], where the author models these on Kan complexes as models for $\infty$-groupoids. $Pr^L$ shall mean the $\infty$-category of large presentable $\infty$-categories, while $Cat_{\infty}$ will stand for the $\infty$-category of all large $\infty$-categories. For an $\infty$-category $\mathcal{C}$, by $\mathcal{P}(\mathcal{C})$ we mean the $\infty$-category of presheaves of spaces on $\mathcal{C}$. By $\mathcal{P}(S)$ we shall mean the $\infty$-category of spaces on $Sm_S$. 

We will use the notation $\mathcal{H}^{X}$ for the full subcategory of $\mathcal{P}$ consisting of $X$-local and nisnevich-local objects. The corresponding localization functors will be denoted $L_{X}, L_{nis}, L_{xmot}$, etc.  

By $\pi_0$ (and similarly $\pi_n$), we shall always mean the presheaf of homotopy groups. This is also the $0$th truncation functor for $\mathcal{P}(S)$. $\pi_n^{nis}$ will denote the homotopy groups internal to the nisnevich topos, i.e., the nisnevich sheafification of $\pi_n$. $\pi_i^{X}$ will, as usual, stand for $\pi_i^{nis}L_{xmot}$.

The stabilization of $\mathcal{H}^-$ will be denoted by $\mathcal{SH}_{S^1}^-$ for $-=pre, nis, \mathbb{A}^1, \mathbb{P}^1$, etc., with the standard adjunction $$\Sigma^\infty_{S^1}: \mathcal{H}_\bullet^-\leftrightarrows \mathcal{SH}^-_{S^1}:\Omega^\infty_{S^1}.$$ 

We will use italics $X,Y,Z, R,S,T$ for schemes, calligraphic letters $\mathcal{C},\mathcal{D}$, etc., for categories, $\esc{X}, \esc{Y}$, etc., for spaces, $\mal{X}, \mal{Y}$ for $S^1$ spectra.

\section{Unstable\texorpdfstring{ $\mathbb{P}^1$}{p1}-homotopy theory}\label{section 2}
\subsection{Introduction to $\mathbb{P}^1$ homotopy theory:}
Parallel to unstable $\mathbb{A}^1$-homotopy theory, we define unstable $\mathbb{P}^1$-homotopy theory. The derived version of this notion first appeared in [\cite{ayoub2020P1}].

\begin{defn}
    Let $S$ be a scheme and $\esc{X}\in \mathcal{P}(S)$. We say that $\esc{X}$ is projective-local or $\mathbb{P}^1$-local if, for every $U\in Sm_S$, the restriction map $\esc{X}(U)\to \esc{X}(\mathbb{P}^1_U)$ induced by the projection $\mathbb{P}^1_U\to U$ is an equivalence of $\infty$-groupoids. The full subcategory of $\mathbb{P}^1$-local spaces is denoted $L_{\mathbb{P}^1}\mathcal{P}(S)\subset \mathcal{P}(S)$.
\end{defn}
Since the collection $$\mathbb{P}(S):=\{\mathbb{P}^1\times X\to X\}_{{X\in Sm_S}}$$ is small and $\mathcal{P}(S)$ is presentable [\cite{lurie2009higher}, Remark 5.5.3.7], [\cite{lurie2009higher}, Proposition 5.5.4.15] implies that the category $L_{\mathbb{P}^1}\mathcal{P}(S)$ is presentable and a localizing subcategory. Let us denote the corresponding localization functor by $L_{\mathbb{P}^1}$.

Recall that a presheaf $\esc{X}\in \mathcal{P}(S)$ is Nisnevich local if it takes Nisnevich squares to pullback squares of $\infty$-groupoids and sends the empty scheme to the contractible space [\cite{hoyois2015quadratic}, Appendix C2]. We denote the full subcategory of Nisnevich-local spaces by $L_{nis}\mathcal{P}(S)$. When $S$ is qcqs, these are precisely the presheaves that satisfy Nisnevich descent [\cite{hoyois2015quadratic}, Appendix C2]. Hence, when $S$ is a qcqs scheme, the subcategory $L_{nis}\mathcal{P}(S)\subset \mathcal{P}(S)$ is a left-exact, accessible localization with localization functor $L_{nis}$.

Here is the main definition of this paper:
  \begin{defn}
  A space $\esc{X}\in \mathcal{P}(S)$ is a projective motivic space if it is both $\mathbb{P}^1$-local and Nisnevich-local. 
  \end{defn}
  
  The full subcategory of projective motivic spaces is denoted $\mathcal{H}^{\mathbb{P}^1}(S)\subset \mathcal{P}(S)$. Since (representable) schemes satisfy Nisnevich descent, $\mathbb{P}(S)$ is contained in $\mathcal{P}_{nis}(S)$. Thus, $\mathcal{H}^{\mathbb{P}^1}(S)$ is the subcategory of the topos $\mathcal{P}_{nis}(S)$ consisting of $\mathbb{P}(S)$-local objects. By [\cite{lurie2009higher}, Proposition 5.5.4.15], this is a presentable category and an accessible localization of the Nisnevich topos. We denote the corresponding localization functor as $L_{pmot}$. Morphisms in the strongly saturated class generated by $\mathbb{P}(S)$ are called projective motivic equivalences.  
\begin{rem}
    The Topological analog of this construction is as follows. Consider the cohesive $\infty$-topos $$\mathbf{H}:=\mathcal{P}_{open}(SmMan)$$ of sheaves of spaces on the category of real smooth manifolds with the open cover topology. Recall that the $\mathbb{R}^1$-localization of this category recovers the $\infty$-topos of spaces [\cite{dugger1999sheaves}, Remark 3.4.10], i.e., there is a canonical equivalence:
    $$Spc\simeq L_{\mathbb{R}^1}\mathcal{P}_{open}(SmMan).$$
    Since $S^1_\mathbb{R}\simeq \mathbb{P}^1_\mathbb{R}$, the analog of the topological $\mathbb{P}^1$-local category is precisely the localization $$L_{S^1}\mathcal{P}_{open}(SmMan).$$ To the best of my knowledge, this category has never been studied, neither by differential cohomology theorists nor by physicists.
\end{rem}
  
\underline{\textbf{On constructing a model for $L_{pmot}$.}}

Despite its complexity, $\mathbb{A}^1$-homotopy theory has a savior that enables several computations and deductions that would otherwise be impossible from the definition alone. This is because $\mathbb{A}^1$ is an interval object \textup{[\cite{morel19991}, Definition 2.3.1]}, and hence the analog of the classical singular complex formula provides an explicit model for $\mathbb{A}^1$-localization [\cite{morel19991}, Lemma 2.3.20].

\begin{obs}\label[obs]{p1 not interval}
Unfortunately, $\mathbb{P}^1_S$ is not an interval in the sense of Morel-Voevodsky \textup{[\cite{morel19991}, Definition 2.3.1]}.
\end{obs} 
\begin{proof}
   Let us consider this over fields. If $\mathbb{P}^1_k$ were an interval object, there would be maps $\mu: \mathbb{P}^1_k\times_k \mathbb{P}^1_k\to \mathbb{P}^1_k$ and two sections $i_0,i_1\in \mathbb{P}^1_k$ such that 
   \begin{flalign}
       \mu (i_0\times id)=\mu(id\times i_0)=i_0\\
   \mu(i_1\times id )=\mu (id \times i_1)=id
   \end{flalign}
   [see \cite{morel19991}, Definition 2.3.1]. Since $\mathbb{P}^1_k$ is complete, by the rigidity lemma [\cite{mumford1974abelian}, Page 40 in the second edition (43 in the old one)], the condition $\mu(id\times i_0)=i_0$ implies that $\mu$ must factor through the second projection $p_2: \mathbb{P}^1_k\times_k \mathbb{P}^1_k\to \mathbb{P}^1_k$ via a rational function $s:\mathbb{P}^1_k\to \mathbb{P}^1_k$. That is, $\mu(x,y)=s(y)$. But then the condition $\mu (id \times i_1)=id$ will fail. Therefore, in general $\mathbb{P}^1_S$ cannot be promoted to an interval on the site $Sm_S^{nis}$. 
\end{proof}

Thus, the singular construction can no longer be used to produce models of $\mathbb{P}^1$-localization. To be precise, the failure of the interval property prevents the singular construction from being a simplicial colimit. The multiplication map $\mu$ provides the degeneracy in the complex. Otherwise, the complex is just a semisimplicial one. In detail, let 
$$\Delta_{\mathbb{P}}^n:=\underset{n-times }{\mathbb{P}^1\times\cdots \times \mathbb{P}^1}.$$ Choose two sections $0, \infty : S\to \mathbb{P}^1_S$. Inserting either $0$ or $\infty$ in one spot, we may turn $\Delta_{\mathbb{P}_S}^n$ into a co-semisimplicial complex $\Delta^\bullet _{\mathbb{P}_S}$ of spaces. We may try $\colim_{}\esc{X}^{\Delta_{\mathbb{P}_S}^\bullet}$ as a prototype for a $\mathbb{P}^1$-local model for $\esc{X}$. Unfortunately, even though $$\esc{X}\to \colim \esc{X}^{\Delta_{\mathbb{P}_S}^\bullet}$$ is a $\mathbb{P}^1$-local equivalence (being a colimit of $\mathbb{P}^1$-local equivalences  $\esc{X}\to \esc{X}^{\Delta_{\mathbb{P}_S}^n}$), the space  $\colim_{}\esc{X}^{\Delta_{\mathbb{P}_S}^\bullet}$ fails to be a $\mathbb{P}^1$-local object. 

However, the good news is that the absence of degeneracy maps can be addressed by including all possible maps in place of the face maps and adding all projection maps in place of degeneracy maps. So let $\delta^\bullet_{\mathbb{P}_S^1}\times Y$ be the sub-diagram of the slice category $({Sm_S})_{/Y}$ given by $$\delta^n_{\mathbb{P}_S^1}\times Y:=(\mathbb{P}^1_S)^{\times n}\times _S Y$$ with structure maps given by the projections to $Y$. 
\begin{thm}\label{model for P1 localization}
\begin{enumerate}
    \item The $\mathbb{P}^1$ localization functor on $\mathcal{P}(S)$ is given by the following colimit formula: $$X\mapsto  L_{\mathbb{P}_S^1}\esc{X}(Y) =\colim_{(\delta^n_{\mathbb{P}_S^1}\times_S Y)^{op}} \esc{X}(\delta^n_{\mathbb{P}_S^1}\times_S Y)$$
    \item    $L_{\mathbb{P}^1}$ is a locally cartesian localization functor.
    \item $L_{\mathbb{P}^1}$ is cartesian.
\end{enumerate}
\end{thm}
\begin{proof}
 This is a special case of [\cite{hoyois2017six}, Proposition 3.4] applied to $\mathcal{C}=\mathcal{P}(S)$ and $$\mathcal{A}=\{\mathbb{P}^1_S\times _SX\to X\}_{X\in Sm_S}.$$ By definition, $\mathcal{A}$ is closed under pullbacks. It is also evident that in Hoyois's notation $$\mathcal{A}_Y\equiv \delta_{\mathbb{P}_S}^n\times_S Y.$$
\end{proof}
\begin{prop}\label[prop]{model for pmot localization}
        The Nisnevich-local model of $\mathbb{P}^1$-localization can be expressed as $$L_{pmot}\simeq \colim _{n\to \infty }(L_{nis}\circ  L_{\mathbb{P}^1})^n.$$
\end{prop}
\begin{proof}
       Since $\mathbb{P}^1$-localization is defined by a schematic set of morphisms, we observe that $L_{\mathbb{P}^1}\mathcal{P}(S)\subset \mathcal{P}(S)$ is closed under filtered colimits (in fact, all colimits). Moreover, $\mathcal{P}_{nis}(S)\subset\mathcal{P}(S)$ is also closed under filtered colimits, as the Nisnevich topology is a cd topology. For a space, the colimit in the statement is computed by taking the cofinal subdiagram consisting of the even-numbered terms of the long diagram:    $$\esc{X}\to L_{\mathbb{P}^1}\esc{X}\to  L_{nis}L_{\mathbb{P}^1}\esc{X}\to  L_{\mathbb{P}^1}L_{nis}L_{\mathbb{P}^1}\esc{X}\to  L_{nis}L_{\mathbb{P}^1}L_{nis}L_{\mathbb{P}^1}\esc{X}\to \cdots.$$
    Since all the terms in the cofinal $2\mathbb{N}$ sub-diagram are nisnevich local, the colimit is also nisnevich local. The subdiagram consisting of the odd terms, on the other hand, is also cofinal and consists of $\mathbb{P}^1$-local spaces, implying that the colimit is $\mathbb{P}^1$-local. The colimit $ \colim _{n\to \infty }(L_{nis}\circ  L_{A}\circ L_{B})^n\esc{X}$ is thus an $\mathbb{P}^1$-local nisnevich space. Since each morphism in the entire diagram, as drawn above, is either an $\mathbb{P}^1$-equivalence or a nisnevich local equivalence (note that each of these is a projective motivic equivalence), we find that the colimiting morphism $ \esc{X}\to \colim _{n\to \infty }(L_{nis}\circ  L_{\mathbb{P}^1})^n\esc{X}$ is also a projective motivic equivalence, by closure of the saturation class under colimits. 
\end{proof}
Recall that the affine motivic localization functor has an identical presentation $$L_{mot}\simeq \colim_n{(L_{nis}L_{\mathbb{A}^1})^n}.$$ However, since $\mathbb{A}^1$ is an interval object, one has $$L_{\mathbb{A}^1}\esc{X}\simeq |\esc{X}^{\mathbb{A}^1}|.$$ Because $\pi_0$ preserves geometric realizations (in fact, all colimits), and geometric realizations in $1$-categories are computed as the coequalizer of the face maps $d_0, d_1$, it follows that there is a canonical surjection $$\pi_0\esc{X}\to \pi_0L_{\mathbb{A}^1}\esc{X},$$ which induces a surjection $$\pi_0^{nis}\esc{X}\to \pi_0^{nis}L_{mot}\esc{X}.$$ Since $\mathbb{P}^1$ is not an interval object, this mechanism does not apply to our case of $\mathbb{P}^1$ motivic spaces. It turns out that there is a better way to produce this surjection that works more generally. The key is the following lemma. 
\begin{lem}\label[lem]{weak terminal lemma}
Suppose $\mathscr{D}$ is a small indexing (1-)category with a weakly terminal object $1$. Let $F:\mathscr{D}\to \mathscr{C}$ be a diagram of $1$-categories. Then $F(1)\to \colim_{d\in D}Fd$ is an epimorphism.
\end{lem}
\begin{proof}
    There is certainly an epimorphism $$\underset{d\in D}{\bigsqcup}Fd\to \colim_{d\in D}Fd$$ (recall that there is a parallel arrow diagram with codomain $\underset{d\in D}{\bigsqcup}Fd$, whose coequalizer computes the whole colimit $\colim_{d\in D}Fd$ at once [\cite[\href{https://stacks.math.columbia.edu/tag/002P}{Lemma 002P}]{stacks-project}], and that a coequalizer is an epimorphism). For every $d$, by weak terminality of $1$, choose a morphism $j_d: d\to 1$. By the cocone condition for colimits, the composite $$Fd\xrightarrow[]{Fj_d} F1\to \colim_{d\in D}Fd$$ must equal the morphism $Fd\to \colim_{d\in D}Fd$. So the above epimorphism must factor through $F(1)$ via the map $$\underset{d\in D}{\bigsqcup}F(j_d):\underset{d\in D}{\bigsqcup}Fd\to F1,$$ yielding the epimorphism $F(1)\to \colim_{d\in D}Fd$ (by the right cancellation property of epimorphisms in a $1$-category).
\end{proof}

\begin{lem}\label[lem]{presheaf pi0 p1 connectivity}
    $L_{\mathbb{P}^1}$ induces a surjection on path-connected components. That is, for any space $\esc{X}$, the canonical map $$\pi_0\esc{X}\to \pi_0L_{\mathbb{P}^1}\esc{X}$$ is a surjective morphism of presheaves of sets on $Sm_S$.
\end{lem}
\begin{proof}
      Let $\infty: S\to \mathbb{P}^1_S$ be the section at infinity. The section $\infty ^n:S\to \delta^n_{\mathbb{P}_S}$ given by $(\infty,\cdots, \infty)$ induces a section $$j_{\delta^n_\mathbb{P}\times Y}:= \infty ^n\times id_Y: Y\to \delta^n_{\mathbb{P}_S}\times_S Y$$ of the projection $ \delta^n_{\mathbb{P}_S}\times_S Y\to Y$. Thus, $id_Y:Y\to Y$ is a weakly terminal object of the diagram $$\mathcal{A}_Y\equiv (\delta^n_{\mathbb{P}_S}\times_S Y)^{op}.$$ By the lemma above, this yields an epimorphism $\pi_0\esc{X}(Y)\to\colim_{m} \pi_0\esc{X}({{\delta_{\mathbb{P}_S}^m}\times Y})$.

Because $\pi_0$ preserves colimits, we have $$\pi_0L_{\mathbb{P}^1}\esc{X}(Y)\cong\colim_{m} \pi_0\esc{X}({{\delta_{\mathbb{P}_S}^m}\times Y}).$$ By the observation above, this implies that for every $Y\in Sm_S$ there is a canonical surjection $$\pi_0\esc{X}(Y)\to \pi_0L_{\mathbb{P}^1}\esc{X}(Y).$$ By functoriality, the canonical map $\pi_0\esc{X}\to \pi_0L_{\mathbb{P}} \esc{X}$ is surjective. Since sheafification preserves surjections, we obtain a surjection $$\pi_0^{nis}\esc{X}\to \pi_0^{nis}L_{\mathbb{P}^1} \esc{X}.$$
\end{proof}
\begin{rem}
It is clear from the proof that it applies equally to localization at any smooth scheme $A/S$ admitting a rational section $S\to A$.
\end{rem}
\begin{prop}[Unstable $\mathbb{P}^1$-connectivity theorem]\label[prop]{P1 connectivity theorem}
For any space $\esc{X}$, the canonical map $$\pi_0^{nis}\esc{X}\to \pi_0^{nis}L_{pmot}\esc{X}=: \pi_0^{\mathbb{P}^1}\esc{X}$$ is a surjective morphism of nisnevich sheaves of sets on $Sm_S$. Consequently, $L_{pmot}$ preserves nisnevich connectivity.
\end{prop}

\begin{proof}
Applying the left adjoint $L_{nis}$ to the epimorphism in the above lemma \Cref{presheaf pi0 p1 connectivity} yields the surjectivity of $$\pi_0^{nis}\esc{X}\to \pi_0^{nis}L_{\mathbb{P}^1} \esc{X}.$$ But since $\pi_0^{nis}L_{nis}\cong \pi_0^{nis}$, this surjection coincides with \begin{flalign}\label{an eqn}
    \pi_0^{nis}\esc{X}\to \pi_0^{nis}L_{nis}L_{\mathbb{P}^1} \esc{X}.
\end{flalign} 

    Because $\pi_0^{nis}$ preserves Nisnevich-local colimits, we have $$\pi_0^{nis}L_{pmot}\simeq \pi_0^{nis}\colim_{n}(L_{nis}L_{\mathbb{P}^1})^n\simeq \colim_{n}\pi_0^{nis}(L_{nis}L_{\mathbb{P}^1})^n,$$ where the transition maps are obtained by repeated application of \cref{an eqn}. Since epimorphisms are closed under colimits, we conclude that the canonical map $$\pi_0^{nis}\esc{X}\to \pi_0^{nis}L_{pmot} \esc{X}$$ is a Nisnevich epimorphism. 
    
 The preservation of connectivity follows easily from this epimorphism.
\end{proof}

A localization for which the components of the unit map are effective epimorphisms is called an effective localization. Thus, the above proposition implies that $L_{pmot}$ is an effective localization.
\begin{cor}\label[cor]{S rational point}
 If $S$ is henselian, then every $\mathbb{P}^1$-connected scheme over $S$ admits an $S$-rational point.
\end{cor}
\begin{proof}
    By \Cref{P1 connectivity theorem} and the fact that $S$ is a point for the Nisnevich topology, for every $X/S$ there is a surjection $$X(S)\to \pi_0^{\mathbb{P}^1}X(S)$$ of sets. Clearly, if $\pi_0^{\mathbb{P}^1}X$ is trivial, $X(S)$ must be non-empty.
\end{proof}
\subsection{Examples}\mbox{}

\underline{\textbf{$\mathbb{P}^1$-local spaces}}

\textbf{1. Birational local spaces} 

Recall that a birational local space is a motivic space $\esc{X}\in \mathcal{H}^{\mathbb{A}^1}(S)$ that is local with respect to the set of dense open immersions in $Sm_S$. Since $\mathbb{A}^1\hookrightarrow \mathbb{P}^1$ is dense, any birational local space is also $\mathbb{P}^1$-local. In particular, all birational sheaves are $\mathbb{P}^1$-local. All examples in [\cite{0bat}, Section 4.1] are thus $\mathbb{P}^1$-local. 

\begin{ex}\label{birational local spaces are p1 local}
    For example, over a field, all zero-dimensional schemes, abelian varieties, $\mathbb{A}^1$-rigid proper varieties (such as proper curves of positive genus) are $\mathbb{P}^1$-local.
\end{ex}

\textbf{2. $\mathbb{P}^1$-local (pre)sheaves of sets}
\begin{defn}
    A presheaf $P\in Sm_S$ is said to be $\mathbb{P}^1$-invariant if, for any $X\in Sm_S$, the morphism $P(X)\to P(\mathbb{P}^1_X)$ induced by the projection $\mathbb{P}^1_X\to X$ is an isomorphism.
\end{defn}
\begin{rem}\label[rem]{p1 invariance remark}
    Any section, say $\alpha : S\to \mathbb{P}^1_S$, provides a retract for the morphism $P(X)\to P(\mathbb{P}^1_X)$. To check $\mathbb{P}^1$-invariance of a presheaf $P$ of sets, it suffices to verify that the morphism $\alpha^*:P(\mathbb{P}^1_X)\to  P(X)$ is injective or that the morphism $P(X)\to P(\mathbb{P}^1_X)$ is surjective. Since presheaves of sets are discrete, a presheaf of sets is a projective motivic space if and only if it is a Nisnevich sheaf that is $\mathbb{P}^1$-invariant as a presheaf of sets.
\end{rem}
\textbf{3. $\mathbb{P}^1$ invariant presheaves}
\begin{ex}[$\mathbb{A}^1$ invariant weakly unramified presheaves of sets]\label[ex]{weak unram a1 is p1}
    An $\mathbb{A}^1$-invariant weakly unramified presheaf is automatically $\mathbb{P}^1$-invariant. To see this, recall from \Cref{remark on weak p1 invariance} that it suffices to show that for any $X\in Sm_S$ the induced pullback map $0^*_{pF}:F(\mathbb{P}^1_X)\to F(X)$ along the section $0_p:X\to \mathbb{P}^1_X$ is injective. Choose the embedding $j_\infty: \mathbb{A}^1_X\subset \mathbb{P}^1_X$ as the complement of $\infty$. Then the section $0_p$ factors as $$X\xrightarrow[]{0_a}\mathbb{A}^1_X\xrightarrow{j_\infty} \mathbb{P}^1_X.$$ Applying this to $F$, we find that the map $0^*_{pF}$ factors as $$F(\mathbb{P}^1_X)\xrightarrow{j_\infty^*}F(\mathbb{A}^1_X)\xrightarrow{0_{aF}^*} F(X).$$ Since $F$ is $\mathbb{A}^1$-invariant, the last map in the composite is injective (in fact, an isomorphism), while the first map is injective because of the weak unramifiedness of $F$.  
\end{ex}
\begin{ex}[Chow group of $0$-cycles]  
    By the projection formula [\href{https://stacks.math.columbia.edu/tag/02TX}{Lemma 02TX}], we know that the flat pullback $\mathrm{CH}^0(X)\to \mathrm{CH}^0(\mathbb{P}^1_X)$ (using $\mathrm{CH}^0=\mathrm{CH}_{dim}$) along the projection map $\mathbb{P}^1_X\to X$ is an isomorphism. However, it is not a Nisnevich sheaf. 
\end{ex}
\textbf{4. $\mathbb{P}^1$ invariant sheaves of sets}
\begin{ex}[$\mathbb{A}^1$ invariant nisnevich sheaves of groups are $\mathbb{P}^1$-invariant]\label[ex]{a1 inv nis group is p1 inv}
Since $\mathbb{A}^1$-invariant sheaves of groups are $\mathbb{A}^1$-invariant by definition and weakly unramified by [\cite{bachmann2024strongly}, Corollary 2.4], it follows that any $\mathbb{A}^1$-invariant Nisnevich sheaf of groups is $\mathbb{P}^1$-invariant; see \Cref{weak unram a1 is p1}. In particular, strongly or strictly $\mathbb{A}^1$-invariant Nisnevich sheaves of groups are $\mathbb{P}^1$-invariant. Examples of this type are plentiful: $K^M_n, K^{MW}_n$ and, more generally, higher $\mathbb{A}^1$ homotopy groups $\pi_i^{\mathbb{A}^1}$ ($i\geq 1$) of spaces. 
\end{ex}
\begin{ex}[More generally $\mathbb{A}^1$ invariant nisnevich sheaves of sets are $\mathbb{P}^1$-invariant]\label[ex]{a1 inv nis is p1 inv}
 In fact, there is an elementary way to show that any $\mathbb{A}^1$-invariant Nisnevich sheaf of sets is $\mathbb{P}^1$-invariant. To see this, consider the standard Zariski cover of $\mathbb{P}^1$ by two copies of the affine line:
\[
\xymatrix{
\mathbb{G}_{m,X}\ar[r]^{j_1}\ar@{^(->}[d]_{j_{0}}&\mathbb{A}^1_X\ar@{^(->}[d]^{i_0}&& F(\mathbb{P}^1_X)\ar[r]^{F(i_0)}\ar[d]_{F(i_\infty)}&F(\mathbb{A}^1)\ar[d]^{F(j_1)}\\
\mathbb{A}^1_X\ar@{^(->}[r]_{i_\infty}&\mathbb{P}^1_X&& F(\mathbb{A}^1)\ar[r]_{F(j_0)}&F(\mathbb{G}_m)
}
\]

 By the Nisnevich sheaf condition, $F(\mathbb{P}^1)$ is given by the pullback square on the right. Let $p:\mathbb{A}^1_X\to X$ and $q:\mathbb{P}^1_X\to X$ be the projection maps. Then $pj_0=pj_1$, so $Fj_0Fp=Fj_1Fp$. Because $F$ is $\mathbb{A}^1$-invariant, $F(p)$ is an isomorphism. Thus $Fj_0=Fj_1$. Moreover, the inclusion $s_1: X\to \mathbb{G}_{m,X}$ gives a section $j_0s_1$ to $p$. Therefore, $Fs_1Fj_0$ is an isomorphism. Thus $Fj_0$ is injective. So the pullback square on the right is a kernel pair of a monomorphism. It follows that $F(i_0)$ is an isomorphism and thus so is $Fq: FX\to F\mathbb{P}^1_X$.
\end{ex}
\begin{rem}
    Additionally, we know that the Zariski sheafification of an $\mathbb{A}^1$-invariant presheaf with transfers is $\mathbb{A}^1$-invariant [Theorem 22.3 of \textit{loc. cit.}]. Since these are also Nisnevich sheaves [Theorem 22.2 of \textit{loc. cit.}], the Zariski sheafification of an $\mathbb{A}^1$-invariant presheaf with transfers is a $\mathbb{P}^1$-motivic space. 
\end{rem}

\begin{ex}[Quasicoherent sheaves of modules]
    By [\cite{ayoub2020P1}, Proposition 3.3], the sheaf associated to a quasicoherent sheaf of modules on the base $S$ is $\mathbb{P}^1$-local.
\end{ex}

\begin{ex}[The Cohomological Brauer group]
Over a regular scheme, the cohomological Brauer group, defined as $H^2_{et}(-,\mathbb{G}_m)$, is $\mathbb{P}^1$-local [\cite{colliot2019brauer}, Corollary 5.1.4]. That it is a Zariski sheaf over a regular base follows from [\cite{colliot2019brauer}, Theorem 3.2.2 and the paragraphs following that]. To see that it satisfies the Nisnevich sheaf condition, let the left square below $$\begin{matrix} W & \xhookrightarrow{j'} & Y &&&&&Br(X) & \xhookrightarrow{j^*} & Br(U) \\ \downarrow{q} & & \downarrow{p} &&&&&\downarrow{p^*} & & \downarrow{q^*} \\ U & \xhookrightarrow{j} & X  &&&&&Br(Y) & \xhookrightarrow{j'^*} & Br(W)\end{matrix}$$ be a Nisnevich square, where $p$ is etale and $j$ is an open immersion. We have to show that the right square is a pullback square. Since $X$ is regular, $j^*$ is injective [see Theorem 3.5.4 of loc.cit]. So it suffices to show that the right-hand square is a weak pullback square. To see this, first observe that $H_{X\setminus U}^3(X_{et}, \mathbb{G}_m)$ and $H_{Y\setminus W}^3(Y_{et}, \mathbb{G}_m)$ are isomorphic by excision for etale cohomology [\cite{milne1980etale}, Theorem 9.7]. Then the long exact sequence of etale cohomology with support [\cite{milne1980etale}, Theorem 9.4] for each open immersion yields the comparison of exact sequences:
$$\begin{matrix}
    0 \rightarrow &H^2(X_{et}, \mathbb{G}_m)\xrightarrow{j^*}& H^2(U_{et}, \mathbb{G}_m)\xrightarrow{\partial}&H_{X\setminus U}^3(X_{et},\mathbb{G}_m)\\
   & \downarrow & \downarrow&\downarrow{\simeq }\\
    0 \rightarrow &H^2(Y_{et}, \mathbb{G}_m)\xrightarrow{j^*}& H^2(W_{et}, \mathbb{G}_m)\xrightarrow{\partial}&H_{Y\setminus W}^3(Y_{et},\mathbb{G}_m)
\end{matrix}$$
A quick diagram chase then yields the desired weak pullback condition.
\end{ex}
\begin{rem}\label[rem]{p1 inv but not a1 inv}
    The Brauer group is an example of a non-$\mathbb{A}^1$-invariant unramified Nisnevich sheaf that is also $\mathbb{P}^1$-invariant.
\end{rem}

\textbf{5. $\mathbb{P}^1$-rigid schemes}

\begin{defn}
    A scheme $P/S$ is said to be $\mathbb{P}^1$-rigid if it is $\mathbb{P}^1$-invariant as a presheaf of sets, i.e., if it induces an isomorphism on sections over the projection map $\mathbb{P}^1_X\to X$, where $X$ is a scheme smooth over $S$. By \Cref{p1 invariance remark}, this is the case iff $$Sch_S(X,P)\to Sch_S(\mathbb{P}^1_X,P) $$ is surjective, iff $$\alpha^*: Sch_S(\mathbb{P}^1_X,P)\to  Sch_S(X,P)$$ is injective for any section $\alpha : S\to \mathbb{P}^1_S$. Since schemes satisfy Nisnevich descent, this is equivalent to $P$ being a projective motivic space.
\end{defn}

\begin{ex}[The affine line]\label[ex]{a1 is p1 local}
    Since the global sections of the projective line $\mathbb{P}^1_X$ coincide with those of $X$ [\cite[\href{https://stacks.math.columbia.edu/tag/01XX}{Lemma 01XX}]{stacks-project}], we conclude that $\mathbb{A}^1_S$ is $\mathbb{P}^1$-local over $Sm_S$. This implies that any open subscheme of $\mathbb{A}^1_S$ (for example, $\mathbb{G}_m$) is also $\mathbb{P}^1$-invariant.. 
\end{ex}
\begin{ex}[Affine schemes]\label[ex]{affine algebras are projective local}
Let $P=\underline{Spec}_S\mathcal{E}$ be an affine morphism over $S$ [\cite[\href{https://stacks.math.columbia.edu/tag/01S8}{Lemma 01S8}]{stacks-project}]. We find that $$ Sch_S(\mathbb{P}^1_X,P)= Sch_S(\mathbb{P}^1_X,\underline{Spec}_S\mathcal{E})\cong QCohAlg_{\mathcal{O}_S}(\mathcal{E}, \Gamma_S(\mathbb{P}^1_X)).$$ As noted above, $\Gamma_S (\mathbb{P}^1_X)=\Gamma_S(X)$. It follows that $$ Sch_S(\mathbb{P}^1_X,P)\cong
QCohAlg_{\mathcal{O}_S}(\mathcal{E}, \Gamma(X))\cong Sch_S(X, \underline{Spec}_S\mathcal{E})= Sch_S(X, P).$$ If our base $S$ is semi-separated, then any affine scheme (over $\mathbb{Z}$) over $S$ is an affine morphism, so all affine schemes are $\mathbb{P}^1$-local over $S$. Since all schemes satisfy Nisnevich descent, all of these are in fact projective motivic spaces. It is a little surprising that all affine algebras over $\mathbb{Z}$ (or a field $k$), smooth or not, are $\mathbb{P}^1$-local over $\mathbb{Z}$ (or $k$). This is a stark difference from the $\mathbb{A}^1$-homotopical world. 
\end{ex}
\begin{ex}[birationally rigid schemes]
From the example above \Cref{birational local spaces are p1 local}, it follows that birationally rigid schemes are, in particular, $\mathbb{P}^1$-rigid. For example, Abelian varieties over fields are birationally rigid [\cite{0bat}, Examples 4.1.3] and hence are $\mathbb{P}^1$-local in particular.
\end{ex}
\begin{rem}
    It was shown in loc. cit. that birationally local spaces are not just $\mathbb{P}^1$-local but also $\mathbb{A}^1$-local. However, this is not always true. In fact, all non-birational examples above are not $\mathbb{A}^1$-local. Most affine schemes, in fact, not even $\mathbb{A} ^1$ itself is $\mathbb{A}^1$-local.
\end{rem}
\begin{ex}(subscheme of a $\mathbb{P}^1$-rigid scheme)\label[ex]{sub of p1 rigid id p1 rigid}
Suppose $f: Y\to X$ is a monomorphism (for example, an immersion) of schemes such that $X$ is $\mathbb{P}^1$-local. Then $Y$ is $\mathbb{P}^1$-local as well. To see this, let $x: \mathbb{P}^1_S\to Y$ be an $S$-morphism. By the $\mathbb{P}^1$-rigidity of $X$, the composite $fx: \mathbb{P}^1_S\to X$ lifts through the projection $\mathbb{P}^1_S\to S$ via a map $t:S\to X$, i.e., $tp=fx$. Then $x\circ \infty:S\to Y$ provides the required lift. Indeed, $$f(x\circ \infty\circ p)=t\circ p\circ \infty \circ p=t\circ p=fx.$$ Since $f$ is monic, we deduce that $(x\circ \infty )\circ p= x$.
\end{ex}
\begin{ex}(Quasi affines)
    From \Cref{sub of p1 rigid id p1 rigid} and \Cref{affine algebras are projective local}, it follows immediately that any quasi-affine (over $S$) is projective local. In particular, all affine spheres $\mathbb{A}^n\setminus 0$ are $\mathbb{P}^1$-motivic spaces.
\end{ex}

\textbf{6. Projective motivic stacks}

\begin{ex}[The moduli stack of finite étale schemes]\label[ex]{fet stack}
  We know that $X\mapsto FEt_X$ is a Nisnevich sheaf of categories. In fact, by effective fpqc descent for affine morphisms [\cite[\href{https://stacks.math.columbia.edu/tag/0244}{Section 0244}]{stacks-project}], the fpqc-localness of morphisms being finite [\cite[\href{https://stacks.math.columbia.edu/tag/02LA}{Lemma 02LA}]{stacks-project}] and etale [\cite[\href{https://stacks.math.columbia.edu/tag/02VN}{Lemma 02VN}]{stacks-project}], we find that $FEt$ is an fpqc sheaf.  
  
  To see why this is $\mathbb{P}^1$-local, let the $Y$ be a connected scheme with a point $y$ and its associated geometric point $\bar{y}$. The projective map $p_Y:\mathbb{P}^1_Y\to Y$ satisfies the condition in the Homotopy sequence [\cite[\href{https://stacks.math.columbia.edu/tag/0BUM}{Section 0BUM}]{stacks-project}], yielding an exact sequence $$\pi_1^{et}(\mathbb{P}^1_{\bar{y}})\to \pi_1^{et}(\mathbb{P}^1_Y)\to \pi_1^{et}(Y)\to 1.$$ But since $\bar{y}$ is algebraically closed, $\pi_1^{et}(\mathbb{P}^1_{\bar{y}})=1$ [\cite{milne1980etale}, Example 3.2]. Consequently, we obtain an isomorphism $\pi_1^{et}(\mathbb{P}^1_Y)\cong \pi_1^{et}(Y)$. Since $FEt$ is the Galois category of the corresponding fundamental group [\cite[\href{https://stacks.math.columbia.edu/tag/0BND}{Theorem 0BND}]{stacks-project}], we deduce that the pull back functor $$p_Y^*:FEt_{Y}\to FEt_{\mathbb{P}^1_Y}$$ is an equivalence. By the sheaf condition (or by [\cite[\href{https://stacks.math.columbia.edu/tag/0EK3}{Section 0EK3}]{stacks-project}]), this holds for any scheme, with finitely many connected components. 

  Let $S$ be a scheme with finitely many connected components and let $$\mathrm{Fet}:=FEt^{\simeq}_{|Sm_S}$$ be the internal maximal subgroupoid, defined sectionwise. Because taking the core commutes with limits, it follows that $\mathrm{Fet}$ is a nisnevich sheaf of spaces. $\mathbb{P}^1$-invariance follows immediately from that of $FEt$. Thus $\mathrm{Fet}$ is a projective motivic space.
 \end{ex}
\begin{rem}
All the examples of $\mathbb{P}^1$-local spaces mentioned in this subsection are, in fact, $\mathbb{P}^n$-local for every $n\geq 0$. It is an interesting question whether every $\mathbb{P}^1$-local Nisnevich-local space is $\mathbb{P}^n$-local for every $n$. The stable version of this question was raised in [\cite{binda2023logarithmic}, Question 4.0.4.]. This holds if, in addition to $\mathbb{P}^1$-localization, an $\mathbb{A}^1$-localization is imposed (see \Cref{a1 p1 is pn}), or for $n=2$ when the Nisnevich topology is replaced with the cdp topology (see \Cref{cdp p1 is p2}). The general case is unclear to me at this point, and I leave it open. 
\end{rem}
\underline{\textbf{$\mathbb{P}^1$-motivic equivalences and $\mathbb{P}^1$-contractible spaces}}

As usual, we say that an $S$-space $\esc{X}$ is $\mathbb{P}^1$-contractible if the canonical map $\esc{X}\to S$ is a $\mathbb{P}^1$-motivic equivalence.
\begin{prop}
    Over a field $k$, the only $\mathbb{P}^1$-contractible scheme of dimension $0$ is $k$ itself.
\end{prop}
\begin{proof}
    This follows immediately from the fact that all zero-dimensional schemes over a field $k$ are $\mathbb{P}^1$-local (\Cref{birational local spaces are p1 local}).
\end{proof}

\begin{prop}
    The only $\mathbb{P}^1$-connected smooth proper curve over a field $k$ is $\mathbb{P}^1_k$. It is thus the only $\mathbb{P}^1$-contractible curve.
\end{prop}
\begin{proof}
  Obviously, $\mathbb{P}^1_k$ is a $\mathbb{P}^1$-contractible smooth proper curve. To see the converse, let $X$ be such a curve. Since $X$ is $\mathbb{P}^1$-connected, by \Cref{S rational point}, $X$ has a $k$-rational point. Also, $X$ must be of genus $0$ by \Cref{birational local spaces are p1 local}. But the only smooth proper curve of genus zero with a rational point is $\mathbb{P}^1_k$ itself.
\end{proof}
\begin{ex}
    A Nisnevich locally trivial $\mathbb{P}^1$-bundle (or more generally, a $(\mathbb{P}^1)^n$-bundle) is a $\mathbb{P}^1$-equivalence. For example, projective bundles of rank $1$ are $\mathbb{P}^1$-equivalences. In fact, every Nisnevich locally trivial $\mathbb{P}^1$-bundle arises this way. Indeed, the short exact sequence $$0\to \mathbb{G}_m\to GL_2\to PGL_2\to 0$$ induces a long exact sequence of cohomology groups:
    $$H^1_{nis}(X,GL_2)\xrightarrow{p} H^1_{nis}(X,PGL_2)\xrightarrow{\partial} H^2_{nis}(X,\mathbb{G}_m).$$ By the Gersten complex of $\mathbb{G}_m$, we know that $H^2_{nis}(X,\mathbb{G}_m)=0.$ Hence, the map $p$ is surjective. But the domain of $p$ classifies rank $2$ vector bundles, whereas the target classifies $\mathbb{P}^1$-bundles.
\end{ex}

\begin{ex}
    For every $n\geq 0$, the Hirzebruch surface $$F_n:=\mathbb{P}_{\mathbb{P}^1}(\mathcal{O}(-n)\oplus \mathcal{O})$$ is $\mathbb{P}^1$-contractible because it is a $\mathbb{P}^1$-bundle over the $\mathbb{P}^1$-contractible space $\mathbb{P}^1$. This stands in stark contrast to the $\mathbb{A}^1$-homotopical classification of surfaces. Indeed, over a field $k$ of characteristic zero, the only $\mathbb{A}^1$-contractible smooth surface is $\mathbb{A}^2_k$ [\cite{dubouloz2015familiesmathbba1contractibleaffinethreefolds}, Theorem 3.12.].
\end{ex}
\begin{ex}
    The projectivized tangent bundle of a smooth surface is a $\mathbb{P}^1$-equivalence. Thus, from the example above, the projectivized tangent bundles of Hirzebruch surfaces are $\mathbb{P}^1$-contractible threefolds. Similarly, by repeatedly constructing $\mathbb{P}^1$-bundles, it is clear that for every $n>1$ there exist non-unique $\mathbb{P}^1$-contractible proper varieties of dimension $n$.
\end{ex}

\begin{rem}
  The general method for classifying projective motivic equivalences via $\mathbb{P}^1$-contractibility of fibers is not feasible, as in the $\mathbb{A}^1$-local case [\cite{dubouloz2015familiesmathbba1contractibleaffinethreefolds}, Proposition 1.15]. The issue is that the proof of the $\mathbb{A}^1$-local characterization (as in \textit{loc.cit.}) relies on the localization sequence, which fails in the $\mathbb{P}^1$-local setting (see \S\ref{fail gluing}).
\end{rem}
 \subsection{Intersection with the \texorpdfstring{$\mathbb{A}^1$}{a1}-homotopy category}\mbox{}
 
To a practitioner of (affine) Motivic homotopy theory, or even to a classical homotopy theorist, it will sound absurd to talk about the contraction of the sphere without imposing the interval homotopy property. In classical topology, in fact, it is very standard to localize the sphere after imposing the interval homotopy property. These objects are precisely the discrete objects. Motivically, $\mathbb{P}^1$ plays the correct `motivic role' of a sphere, not $S^1$. In fact, this suggests an interesting way to think about objects in the intersection $\mathcal{H}^{\mathbb{A}^1}(S)\bigcap \mathcal{H}^{\mathbb{P}^1}(S)$. To discuss this, let us first recall what is usually meant by discrete objects in the categorical sense. 

Let us suppose that $\mathcal{C}$ is a presentable $\infty$-category. By [\cite{lurie2017higher}, Example 4.8.1.20], it is presentably tensored over $Spc$ (in fact, $\mathcal{C}\simeq Spc\otimes \mathcal{C}$), and hence `contains' a copy of the topological space $S^1$. More precisely, for any $X\in \mathcal{C}$, there is an object $S^1\otimes X$ given by the pushout $*\leftarrow X\rightarrow *$. An object $Y\in \mathcal{C}$ is said to be `categorically/homotopically' discrete in $\mathcal{C}$ if it is local with respect to the projection maps $S^1\otimes X\to X$. These are the simplest homotopy types in $\mathcal{C}$ in the sense that they contain no higher homotopical information. Also note that, since topologically $S^1\simeq \mathbb{P}^1_\mathbb{R}$ the category of discrete objects of the $\infty$-category $Top$ is precisely the $\mathbb{P}^1_\mathbb{R}$ localization of $Top\simeq Spc$. In other words, $$L_{\mathbb{P}^1_\mathbb{R}}Top\simeq \mathrm{Disc}(Top)\simeq Sets.$$ 

In the motivic category $\mathcal{H}^{\mathbb{A}^1}(S)$, however, the topological circle is not the correct motivic sphere. The correct motivic `circle' is rather $\mathbb{P}^1_S$. Thus, when working with $\mathcal{H}^{\mathbb{A}^1}$, the right notion of `Motivically discrete' motivic spaces is given by local objects with respect to  $\mathbb{P}^1$, which is precisely $\mathcal{H}^{\mathbb{A}^1}(S)\bigcap \mathcal{H}^{\mathbb{P}^1}(S)$.
\begin{defn}
    We say that a motivic space is motivically discrete if it is $\mathbb{P}^1$-local. We denote by $$\mathrm{MDisc}(S):=\mathcal{H}^{\mathbb{A}^1}(S)\bigcap \mathcal{H}^{\mathbb{P}^1}(S)$$ the full subcategory of motivically discrete motivic spaces.
\end{defn}  
\begin{rem}
    It is clear that $\mathrm{MDisc}(S)$ is the localization of $\mathcal{P}_{nis}(S)$ at the set of $\mathbb{A}^1$-projections and $\mathbb{P}^1$-projections. Since the union of these sets is small, $\mathrm{MDisc}(S)\subset \mathcal{P}(S)$ is an accessible localization.
\end{rem}
\begin{rem}
    This is precisely the $S^{2,1}$ nullification of the ordinary motivic homotopy category introduced and studied in [\cite{asok2023p}, \S3].
\end{rem}
It is clear that all motivic spheres of positive simplicial weight are contractible in $\mathrm{MDisc}$. In fact, all projective spaces are contractible in $\mathrm{MDisc}$:
\begin{lem}\label[lem]{a1 p1 is pn}
  For all $n\geq 0$,
    $$\mathrm{MDisc}=L_{\mathbb{P}^{1}}{\mathcal{H}^{\mathbb{A}^1}}\subset L_{\mathbb{P}^{n}}\mathcal{H}^{\mathbb{A}^1}.$$
\end{lem}
\begin{proof}
    By [\cite{morel19991}, Corollary 3.2.18] we have a cofiber sequence:
    $$\mathbb{P}^{n-1}\to \mathbb{P}^n\to ({\mathbb{P}^1})^{\wedge n}.$$ 
    The cofiber is clearly contractible in $\mathrm{MDisc}$. The claim then follows by induction.
\end{proof}

\begin{rem}\label[rem]{cdp p1 is p2}
Cdh locally, or more generally, cdp locally, we expect this to hold without the $\mathbb{A}^1$ localization. Indeed, we can prove the cdp-local non $\mathbb{A}^1$-local claim for $n=2$:
\begin{proof}($L_{\mathbb{P}^1}\mathcal{P}_{cdp}\subset L_{\mathbb{P}^2}\mathcal{P}_{cdp}$).

Let $W$ be the blowup of $\mathbb{P}^1\times \mathbb{P}^1$ at $(0,0)$ giving rise to the cdp pushout square:
    \[
    \xymatrix{
    E\ar@{^(->}[r] \ar[d]&W \ar[d]\\
    (0,0)\ar@{^(->}[r]&\mathbb{P}^1\times \mathbb{P}^1
    }
    \]
   Since the exceptional fiber of this blow-up is a divisor isomorphic to $\mathbb{P}^1$, the left vertical map is a $\mathbb{P}^1$-cdp equivalence, and therefore so is the right one. In other words, $W$ is contractible in $L_{\mathbb{P}^1}\mathcal{P}_{cdp}$. 

    It is well known that $W$ is also the total space of the blow-up of $\mathbb{P}^2$ at two points, say $\{0,1\}$, with exceptional fiber isomorphic to the disjoint union of two copies of $\mathbb{P}^1$. One may now consider a similar cdp pushout and conclude that $W\to \mathbb{P}^2$ is a $\mathbb{P}^1$-cdp equivalence. Thus, we win by the previous paragraph. \end{proof}
\end{rem}
It follows from \Cref{a1 inv nis is p1 inv} that $$\emptyset\neq {Sh^{\mathbb{A}^1}_S}=\mathrm{Disc}(\mathcal{H}^{\mathbb{A}^1}(S)))\subset \mathcal{H}^{\mathbb{A}^1}(S)\bigcap \mathcal{H}^{\mathbb{P}^1}(S)=:\mathrm{MDisc}(S).$$ That is, topologically or categorically discrete motivic spaces are motivically discrete motivic spaces as well. However, the converse is not true. In fact, motivically discrete spaces need not even be categorically discrete, unlike the case of manifolds (where the resulting category is equivalent to $Sets$). To elaborate on this point, let us observe that $\mathrm{MDisc}(S)$ contains an interesting subcategory, namely the Birational Motivic homotopy category $\mathcal{H}^0(S)$ of [\cite{0bat}, \cite{sfbat}]. 

Recall that over an arbitrary base scheme $S$, this is defined as $$\mathcal{H}^0(S): =L_{dense}\mathcal{H}^{\mathbb{A}^1}(S)$$ (also denoted by  $\mathcal{H}^{b\mathbb{A}^1}(S)$ in [\cite{0bat}]), the localization of the motivic homotopy category at the set of dense open immersions in $Sm_S$. (One of the main observations of [\cite{0bat}] is that when $S$ is qcqs, there is an easier construction, namely $$\mathcal{H}^b(S):=L_{dense}L_\Sigma\mathcal{P}(S),$$ the localization of the category of $\Sigma$-sheaves over $S$ at the set of dense open immersions in $Sm_S$ [see Corollary 2.2.9 of \textit{loc.cit.}.)

\begin{prop}
    Let $S$ be a scheme. Then $$\mathcal{H}^0(S)=\mathcal{H}^{b\mathbb{A}^1}(S)\subset \mathrm{ MDisc}(S).$$ When $S$ is qcqs, $$L_{dense}L_\Sigma\mathcal{P}(S)=:\mathcal{H}^b(S)\subset \mathrm{MDisc}(S).$$
\end{prop}
\begin{proof}
    Obviously, the dense open immersion $\mathbb{A}^1_S\hookrightarrow \mathbb{P}^1_S$ guarantees that a dense local $\mathbb{A}^1$ local presheaf is automatically $\mathbb{P}^1$ local. In other words, $$\mathcal{H}^{0}(S)\subset \mathcal{H}^{\mathbb{A}^1}(S)\bigcap \mathcal{H}^{\mathbb{P}^1}(S).$$ When $S$ is qcqs, we know that $\mathcal{H}^b(S)=\mathcal{H}^0(S)$ [\cite{sfbat}, Corollary 2.3.7 ; or, \cite{0bat} Corollary 2.2.12.]
\end{proof}

One of the key results of [\cite{0bat}, Theorem 3.3.11] is that for connected spaces over perfect fields, this inclusion is an equivalence:
\begin{thm}\label{connected a1 p1 is bir}
    Let $k$ be a perfect field. Then the inclusion $\mathcal{H}^b(k)\subset \mathrm{MDisc}(k)\equiv \mathcal{H}^{\mathbb{A}^1}(k)\bigcap\mathcal{H}^{\mathbb{P}^1}(k)$ restricts to an equality $$\mathcal{H}^b(k)_{\geq 1}=\mathrm{MDisc}(k)_{\geq 1}$$ where $(-)_{\geq 1}$ denotes $(-)\bigcap\mathcal{P}_{\geq 1}^{nis}$.
\end{thm}
 The above theorem is not true in general for non-connected spaces. For example, $Shv^{\mathbb{A}^1}_S=\mathrm{Disc}(\mathcal{H}^{\mathbb{A}^1}(S))\subset \mathcal{H}^{\mathbb{A}^1}(S)\bigcap \mathcal{H}^{\mathbb{P}^1}(S)$ (see \Cref{a1 inv nis is p1 inv}) while $Shv^{\mathbb{A}^1}_S\not \subset \mathcal{H}^b(S)$ (for example $\mathbb{G}_m$ is $\mathbb{A}^1$ local discrete space but is clearly not birational local).
 
 Delooping the above theorem yields a similar equivalence, though for `correct' group-like spaces. For this, let us make the following definition, similar to Morel's notion of strong $\mathbb{A}^1$-invariance:
\begin{defn}\label[defn]{strong p1 monoid space}
    A monoid space $\esc{G}\in \mathcal{P}_{nis}(S)$ is said to be strongly $\mathbb{P}^1$-invariant (resp. strongly $\mathbb{A}^1$-invariant) if $\mathrm{B}^i_{nis}G$ is $\mathbb{P}^1$-local (resp. $\mathbb{A}^1$-local) for $i=0, 1$. 
\end{defn}
We denote the category of strongly $\mathbb{P}^1$-invariant monoid spaces as $\mathrm{Mon}_{\mathbb{P}}(\mathcal{H}^{\mathbb{P}^1})$ and, similarly, the category of strongly $\mathbb{A}^1$-invariant monoid spaces as $\mathrm{Mon}_{\mathbb{A}}(\mathcal{H}^{\mathbb{A}^1})$. Since $$\mathcal{H}^b(k)\subset \mathcal{H}^{\mathbb{A}^1}(k)\bigcap\mathcal{H}^{\mathbb{P}^1}(k)$$ and the bar construction of birational local monoids is again birational local [\cite{0bat}, Corollary 3.2.7], it follows that $$\mathrm{Mon}(\mathcal{H}^b(k))\subset \mathrm{Mon}_{\mathbb{P}}(\mathcal{H}^{\mathbb{P}^1}(k))\bigcap \mathrm{Mon}_{\mathbb{A}}(\mathcal{H}^{\mathbb{A}^1}(k)).$$
\begin{thm}\label{strong a1 p1 groups are bir}
Let $k$ be a perfect field. A grouplike monoid space $\esc{G}\in \mathrm{Mon}(\mathcal{P}_{nis}(k))^{gp}$ is birational local if and only if it is strongly $\mathbb{A}^1$-invariant and strongly $\mathbb{P}^1$-invariant. In short $$\mathrm{Mon}(\mathcal{H}^b(k))^{gp}=\mathrm{Mon}_{\mathbb{P}}(\mathcal{H}^{\mathbb{P}^1}(k))^{gp}\bigcap \mathrm{Mon}_{\mathbb{A}}(\mathcal{H}^{\mathbb{A}^1}(k))^{gp}.$$
\end{thm}
\begin{proof}
    By the discussion above, it remains to show that a grouplike strongly $\mathbb{A}^1$-invariant, strongly $\mathbb{P}^1$-invariant monoid space $\esc{M}$ is birational local. Since $\esc{G}$ is grouplike, we have $\esc{G}\simeq \Omega\mathrm{B}_{nis}\esc{G}$, where $\mathrm{B}_{nis}\esc{G}$ is $\mathbb{A}^1$-local and $\mathbb{P}^1$-local. Since $\mathrm{B}_{nis}\esc{G}$ is connected, \Cref{connected a1 p1 is bir} implies that $\mathrm{B}_{nis}\esc{G}$ is birational local. We then have the fiber sequence $$\esc{G}(\simeq \Omega\mathrm{B}_{nis}\esc{G})\to *\to \mathrm{B}_{nis}\esc{G}$$ with base and total space birational local. Since $\mathcal{H}^b\subset \mathcal{P}_{nis}$ is closed under limits, the claim follows.
\end{proof}
In particular, a strongly $\mathbb{A}^1$ invariant sheaf of groups is birational local iff it is strongly $\mathbb{P}^1$ invariant.
To conclude this subsection, let us point out that even though motivic discrete groups are projective motivic spaces [see \Cref{a1 inv nis group is p1 inv}], in general, strongly/strictly $\mathbb{A}^1$ invariant grouplike monoids need not be strongly/strictly $\mathbb{P}^1$ invariant. Here is a counterexample,
\begin{coex}\label[coex]{Gm is not strong p1 local}
    Recall that since the prersheaves $\mathbb{G}_m$ and $\mathrm{Pic}$ are both $\mathbb{A}^1$ invariant, the classifying space $\mathrm{B}_{nis}\mathbb{G}_m$ is an $\mathbb{A}^1$ local space. However, $\mathrm{Pic} $ is not $\mathbb{P}^1$ invariant. In fact, $$\mathrm{Pic}(\mathbb{P}^1_S)\simeq \mathrm{Pic}(S)\oplus \mathrm{Pic}(S).$$ Therefore $\mathrm{B}_{nis}\mathbb{G}_m$ is not $\mathbb{P}^1$ local. 
\end{coex}
Though the classifying space of $\mathbb{G}_m$ can still be identified with the stack of global sections of line bundles:
\begin{thm}\label{htype of BGm}
    The canonical forgetful map $$[\mathbb{A}^1/\mathbb{G}_m]\to \mathrm{B}_{nis}\mathbb{G}_m$$ is a projective motivic equivalence. 
\end{thm}
\begin{proof}
We will construct a $\mathbb{P}^1$-homotopy, in fact, at the presheaf level. We recognize the quotient stack as the two-sided bar construction $\mathrm{B}^{}(\mathbb{A}^1,\mathbb{G}_m,*)$, and similarly, the classifying space $\mathrm{B}^{nis}\mathbb{G}_m$ as $\mathrm{B}^{nis}(*,\mathbb{G}_m,*)$. The projection and section $*\xrightarrow{0} \mathbb{A}^1\xrightarrow{p} *$ thus provide a splitting $$\mathrm{B}^{}(*,\mathbb{G}_m,*)\xrightarrow{b_0}\mathrm{B}^{}(\mathbb{A}^1,\mathbb{G}_m,*)\xrightarrow{b_p} \mathrm{B}^{}(*,\mathbb{G}_m,*).$$ It remains to construct a $\mathbb{P}^1$-homotopy from the other composite, namely, the trivial endomorphism $b_0\circ b_p$ of $[\mathbb{A}^1/\mathbb{G}_m]$ given by $(\mathcal{L},s)\mapsto (\mathcal{L},0)$, to its identity. Define, $$H:[\mathbb{A}^1/\mathbb{G}_m]\times \mathbb{P}^1\to [\mathbb{A}^1/\mathbb{G}_m]$$ by 
    $$\{(\mathcal{L},s),(\mathcal{K}, s_0,s_\infty)\}\mapsto (\mathcal{L}\otimes\mathcal{K}, s\otimes s_\infty)$$ It is evident that $H(-,0)=b_0\circ b_p$ whereas $H(-, \infty)=id$.
\end{proof}
\begin{rem}
    Note that this is also a motivic equivalence. The proof follows rather formally because the Bar construction preserves motivic equivalences and $\mathbb{A}^1\to *$ is a motivic equivalence. In fact, this motivic equivalence restricts to a motivic equivalence of their $1$-skeleton, namely $T=\mathbb{A}^1/\mathbb{G}_m\to \Sigma \mathbb{G}_m$. This new space $T$ plays the fundamental role of a sphere in affine motivic homotopy theory and is well known to be identified with $\mathbb{P}^1$. $\mathbb{P}^1$ locally these observations fail. 
\end{rem}

\begin{prop}
    There are canonical projective motivic equivalence of spaces $\mathbb{A}^1/\mathbb{G}_m\simeq \Sigma \mathbb{A}^1$.
\end{prop}
\begin{proof}
    This follows from considering the standard Zariski cover of $\mathbb{P}^1$ by two copies of the affine line, twisted by the punctured line:    \[
    \xymatrix{
    \mathbb{A}^1\setminus 0\ar@{^(->}[r]\ar@{^(->}[d]& \mathbb{A}^1\ar@{^(->}[d]\\
    \mathbb{A}^1\ar@{^(->}[r]&\mathbb{P}^1
    }
    \]
    Because this is a Nisnevich cd square, it is cocartesian. Thus, we have an equivalence of Nisnevich cofibers: $$L_{nis}(\mathbb{A}^1/\mathbb{A}^1\setminus 0)\simeq L_{nis}(\mathbb{P}^1/\mathbb{A}^1).$$ Applying the left adjoint $L_{pmot}$, we obtain the desired equivalence, since $L_{pmot}\mathbb{P}^1\simeq *$.
\end{proof}
\begin{defn}[Thom space of a bundle]
    For a vector bundle $p: E\to X$, the projective Thom space $\mathbb{P}\mathrm{Th}(p)$ is defined as $$\mathbb{P}\mathrm{Th}(p):=L_{pmot} E/(E\setminus 0_X),$$ where $0_X$ denotes the zero section of $p$. 
\end{defn}
\begin{cor}
    For any trivial vector bundle $E$ of rank $n$ over $X/S$, one has $$\mathbb{P}\mathrm{Th}(E/X)\simeq L_{pmot}\Sigma^n(\mathbb{A}^{\wedge n}\wedge X_+).$$
\end{cor}
\begin{proof}
    By [\cite{morel19991} Proposition 2.17 (2)], $$\mathrm{Th}(E/X)\simeq^{zar} X_+\wedge T^{\wedge n}.$$ Applying the monoidal functor $L_{pmot}$ and using the above proposition yields the desired result. 
\end{proof}
By Zariski excision, the projective Thom space of a vector bundle $q: V\to X$ is simply the $\mathbb{P}^1$-localization of the Thom space of the projectivization $$\mathbb{P}(V):=\mathbf{P}(V\oplus \mathcal{O}_X)\to X.$$ Thus, $$\mathbb{P}\mathrm{Th}(q)\simeq  \mathbf{P}roj(V\oplus \mathcal{O}_X)/\big(\mathbf{P}roj(V\oplus \mathcal{O}_X)\setminus X\big)=\mathbb{P}(V)/\mathbb{P}(V)\setminus X.$$ However, it is not clear what the correct $\mathbb{P}^1$-local version of the purity isomorphism should be.

\subsection{Limits and Colimits}
In this small sunbsection we present some formal limit colimit properties of the projective motivic homotopy category. These are true for localization at any smooth scheme. 
\begin{cor}\label[cor]{pmot cartesian localization}
    $L_{pmot}$ is cartesian and locally cartesian.
\end{cor}
\begin{proof}
Recall that $L_{nis}$ is left exact and that $L_{\mathbb{P}^1}$ is cartesian and locally cartesian \Cref{model for P1 localization}. Since being cartesian is stable under filtered colimits, the cartesianness of $L_{pmot}$ follows from the formulae for $L_{pmot}$ given above. Since both functors are locally cartesian and $\mathcal{P}(S)$ has universal colimits, the local cartesianness follows as well from the same model above and the fact that pullbacks are stable under filtered colimits.
\end{proof}
\begin{lem}
    $\mathcal{P}_{zar}(S)$ is generated under sifted colimits by $X\in Sm_S$ such that $X$ is affine. 
\end{lem}
\begin{proof}
        Since $\mathcal{P}_{zar}(S)\subset \mathcal{P}_{\Sigma}(S)$ and $\mathcal{P}_{\Sigma}(S)$ is the sifted colimit completion of $Sm_S$, the Zariski topos is generated under sifted colimits by $X\in Sm_S$. But a Zariski cover of $X$ by affine opens is such that every pairwise intersection is a separated scheme. Thus $\mathcal{P}_{zar}$ is generated under sifted colimits by separated schemes over $S$. But separated schemes, by definition, have Zariski covers by affines such that every pairwise intersection is an affine scheme. So we conclude that $\mathcal{P}_{zar}(S)$ is generated under sifted colimits by affine schemes smooth over $S$.
\end{proof}
The next result is also quite standard, except for its last statement:
\begin{prop}
    $\mathcal{H}^{\mathbb{P}^1}(S)$ is generated under sifted colimits by the projective motivic localizations of smooth affine schemes over $S$ (over $\mathbb{Z}$), i.e., by the set $$\{L_{pmoot}(X\xrightarrow[]{smooth} S)\}_{X=Spec A}.$$ If $S$ is semi-separated (e.g., separated), then these Nisnevich local generators are already $\mathbb{P}^1$-local.
\end{prop}
\begin{proof}
Since the Nisnevich topos, and hence $\mathcal{H}^{\mathbb{P}^1}$, is an accessible localization of $\mathcal{P}_{zar}(S)$, the first statement follows from the lemma above. The second statement follows from Example \Cref{affine algebras are projective local}.  
\end{proof}

\section{Schematic functoriality of \texorpdfstring{ $\mathbb{P}^1$}{{P}1}-motives}
In this section, we discuss the standard functorial behavior of the assignment $$S\mapsto \mathcal{H}^{\mathbb{P}^1}(S).$$ The main point of this short section is to show that this assignment fails to satisfy the six-functor localization sequence.
\subsection{Unstable functors}

It is evident that the $\mathbb{P}^1$-motivic homotopy construction enjoys the unstable functorialities along schematic morphisms. More precisely, for a morphism of schemes $f: S\to T$ the induced functor $f^*: \mathcal{P}(T)\to \mathcal{P}(S)$ preserves $\mathbb{P}^1$-projections and nisnevich squares. It therefore descends to a morphism in $Pr^L$ given by the adjunction $$L_{pmot}f^*: \mathcal{H}^{\mathbb{P}^1}(T)\leftrightarrows \mathcal{H}^{\mathbb{P}^1}(S): f_*.$$ When $p: S\to R$ is smooth the canonical map $Sm_S\to Sm_R$  given by post composition with $p$ induces another adjunction $$L_{pmot}p_\sharp : \mathcal{ H}^{\mathbb{P}^1}(S)\leftrightarrows \mathcal{H}^{\mathbb{P}^1}(R): p^*,$$ in which case $L_{pmot}p^*\simeq p^*$. 
\begin{rem}
    One of the key ingredients of the six-functor formalism for the $\mathbb{A}^1$-motivic homotopy construction is that the pushforward along a closed immersion preserves $\mathbb{A}^1$-motivic equivalences of contractible presheaves (i.e., presheaves that send the empty scheme to a contractible space). In fact, pushforward along any morphism of schemes preserves $\mathbb{A}^1$-homotopy equivalences, while the additional conditions of contractibility and closedness are required only to ensure preservation of Nisnevich equivalences. Nonetheless, this fails for our $\mathbb{P}^1$-localization. The point is that, since $f_*$ is left exact, to ensure the preservation of $X$-homotopy equivalences, it suffices to show that $f_*X$ is contractible. For $X=\mathbb{A}^1$, this follows from the interval property of $\mathbb{A}^1$. In detail, let $f:S\to T$ and suppose $\mu_S:\mathbb{A}_S^1\times_S \mathbb{A}^1_S\to \mathbb{A}^1_S$ is the standard multiplication map. Finally, let $i_S:S\to \mathbb{A}^1_S$ be the zero section. Then the strict $\mathbb{A}^1_T$-homotopy inverse of $p:f_*\mathbb{A}^1_S\to T=f_*S$ is given by $f_*i_S$. Indeed, a strict $\mathbb{A}^1_T$-homotopy between $f_*i_S\circ p$ and the identity can be given by the following composition:
    $$f_*\mathbb{A}^1_S\times \mathbb{A}^1_T\to f_*\mathbb{A}^1_S\times f_*f^*\mathbb{A}^1_T\simeq f_*\mathbb{A}^1_S\times f_*\mathbb{A}^1_S\xrightarrow{f_*\mu}f_*\mathbb{A}^1_S$$ 
    Since $\mathbb{P}^1$ lacks such a multiplication map (\Cref{p1 not interval}), deducing a similar result for the $\mathbb{P}^1$-motivic localization seems harder.
\end{rem}

\subsection{Continuity}
We list out some standard continuity results for $\mathcal{H}^{\mathbb{P}^1}(-)$ that follow from its classical $\mathbb{A}^1$ local counterpart. Most of the proofs can be obtained from [\cite{hoyois2017six}, \S4], see [\cite{sfbat}, \S4] though.
\begin{prop}\label[prop]{bir detect}
       \begin{enumerate}
           \item For a nisnevich cover $f^\alpha: U^\alpha\to S$, the pull back maps $$f^{\alpha,*}: \mathcal{P}(Sm_S)\to \mathcal{P}(Sm_{U^\alpha})$$ detects projective motivic equivalences.
           \item  The presheaf $\mathcal{H}^{\mathbb{P}^1}:Sch^{op}\to Pr^L$ is nisnevich separated. i.e. For a nisnevich cover $f^\alpha: U^\alpha\to S$, the pull back maps $$f^\alpha: \mathcal{H}^{\mathbb{P}^1}(S)\to \mathcal{H}^{\mathbb{P}^1}(U^\alpha)$$ are jointly conservative. 
           \item The presheaf $$\mathcal{H}^{\mathbb{P}^1}(-): Sch\to Pr^L$$ is a Nisnevich sheaf.
       \end{enumerate}  
    \end{prop}
    \begin{proof}
By replacing $\mathbb{A}^1$ projections with $\mathbb{P}^1$ projections, claims 1 and 2 follow verbatim from [\cite{hoyois2017six}, Proposition 4.5], while 3 follows the same way as in Proposition 4.8 in {\textit{loc.cit}}.     \end{proof}
 \begin{thm}\label{continuity of Hb}
   Let $S=\plim S_\alpha$ be a projective system of qcqs schemes with affine transition maps $f_{\alpha\beta}$ and legs $\phi_\alpha: X\to X_\alpha$. Then the pull back assembly map in $Pr^L$ $$\phi=\dcolim_{\alpha} \phi_\alpha^*:  \dcolim_{\alpha}\mathcal{H}^{\mathbb{P}^1}(S_\alpha)\to \mathcal{H}^{\mathbb{P}^1}(S)$$ induced by the left adjoints $\phi_\alpha^*$ is an equivalence of presentable infinity categories. Equivalently, $$\plim_{\alpha} \phi_{\alpha,*}:  \mathcal{H}^b(S)\to \plim_{\alpha}\mathcal{H}^{\mathbb{P}^1}(S_\alpha)$$ is an equivalence in ${Cat}_\infty$.
\end{thm}
\begin{proof}
    This follows verbatim to  [\cite{hoyois2015quadratic}, Proposition C.2] by replacing $\mathbb{A}^1$ projections through $\mathbb{P}^1$ projections. \end{proof}
\subsection{Failure of Gluing}\label{fail gluing}
In this subsection, we demonstrate that the $\mathbb{P}^1$-local version of the motivic fracture square [\cite{morel19991}, Theorem 3.2.21]. This prevents us from establishing the complete six-functor formalism for the assignment $S\mapsto \mathcal{SH}^{\mathbb{P}^1}_{S^1}(S)$.

The key ingredient for producing the localization fracture square [\cite{morel19991}, Theorem 3.2.21] is the following. Let $p: X\to S$ be a smooth morphism, $i:Z\to S$ a closed immersion, and $t: Z\to X$ a partially defined section of $p$. One first constructs the space $h_S(X,t)$ of fractured deviations:
$$h_S(X,t)(Y):=\begin{cases}
   * \text{ if } Y_Z=\emptyset\\
   Hom_S(Y,X)\times _{Hom_Z(Y_Z,X_Z)}* \text{  otherwise}
\end{cases}$$

and then proves that this space is motivically contractible. In fact, it is known that it is Nisnevich locally connective and that $L_{nis}L_{\mathbb{A}^1}L_{nis}h_S(X,t)$ and hence $L_{mot}h_S(X,t)$ is contractible. We intend to show that $L_{pmot}h_S(X,t)$ is not always contractible. We do this by producing an explicit counterexample.
\vspace{.5 cm}
\begin{coex}
Let $S=\mathbb{A}^1_k$, $Z=k\xhookrightarrow{0}S$, and $t:k\xrightarrow[]{(0,0)} X=\mathbb{A}^1_S$. The projection $p:h_S(P,t)\to h_S(S)=*$ is not a projective equivalence.
\end{coex}
\begin{proof}
With the following choices of coordinates, we let $S=Spec k[x]$, $Z=Speck$, $X=Speck[x,y]$, with $Z$ given by the ideal $(x)$ in $k[x]$, the projection map $X\to S$ given by the inclusion $k[x]\subset k[x,y]$, and $t$ given by the ideal $(x,y)$. Let us recall that $X=\mathbb{A}^1_S$ is $\mathbb{P}^1$-local over $S$ (\Cref{affine algebras are projective local}). Thus, it suffices to show that the induced map $$PSh_S(S, X)\to PSh_S(h_S(X,t),X)$$
is not an isomorphism. Of course, this map is injective via the canonical section $S\to h_S(X,t)$. Hence, we shall show that this isn't surjective. Let $a: Y=SpecR\to Spec k[x]=S$ be a smooth $k$-algebra given by a choice $x\mapsto a\in R$. Observe that in that case $X_Z=Spec k[y]$, $Y_Z=Spec(R/aR)$. Clearly, $Y_Z\neq \emptyset$ iff $a$ is not a unit. The trivial map $*: Y_Z\to X_Z$ corresponds to $y\mapsto aR$. So an element of $h_S(X,t)(Y)$ corresponds to an $S$-map $SpecR=Y\to X=Speck[x,y]$ given by $x\mapsto a$ ($S$-map) and $y\mapsto y'$ such that $y'\in Ra$. That is, $h_S(X,t)(a:Y\to S)=Ra$, and the projection $h_S(X,t)\to h_S(S)$ on $Y$-sections is given by the inclusion $Ra\to *$. The $Y$-sections of $X$ are given by $y\mapsto b\in R$. The $S$-morphism $\phi: h_S(X,t)\to X$, given on $Y$-sections by $Ra\ni r \mapsto rb \in R$, cannot come from an $S$-morphism $Y\to X$ (which must be given by a fixed $c\in R$) via the projection $Ra\to *$.
\end{proof}
Corresponding to the above example, let $U$ be the open complement $U=\mathbb{G}_m\hookrightarrow \mathbb{A}^1=S$. The expected localization fracture square is on the left and is the one on the right when $F=X$ from the above example. 
\[
\xymatrix{
j_\sharp j^*F\ar[r]\ar[d]&F\ar[d]&&&h_S(X_U)\ar[r]\ar[d]&h_S(X)\ar[d]\\
h_S(U)\ar[r]& i_*i^*F&&& h_S(U)\ar[r]& i_*X_Z
}
\]
If this were to be cocartesian in $\mathcal{H}^{\mathbb{P}^1}(S)$ we must have a projective motivic equivalence $h_S(X)\bigsqcup_{h_S(X_U)}h_S(U)\to i_*X_Z$. 

Because $X_Z$ is affine and hence $\mathbb{P}^1$-motivic, we know that the space $i_*X_Z$ is projective motivic local. Since $L_{pmot}$ is locally cartesian (\Cref{pmot cartesian localization}), for every $t: S\to i_*X_Z$ the pullback $$\big(h_S(X)\bigsqcup_{h_S(X_U)}h_S(U)\big )\times_{i_*X_Z} S\to S$$ must be a projective motivic equivalence over $S$. But $h_S(X,t)$ is precisely the space $\big(h_S(X)\bigsqcup_{h_S(X_U)}h_S(U)\big )\times_{i_*X_Z} S$. That is, we must have that $h_S(X,t)\to S$ is a $\mathbb{P}^1$-motivic equivalence, which contradicts the counterexample we have shown. 

\section{\texorpdfstring{$\mathbb{P}^1$}{{P}1}-Algebraic topology}

The aim of this section is to carry out a $\mathbb{P}^1$ analog of Morel's program on the strong $\mathbb{A}^1$-invariance of $\mathbb{A}^1$-homotopy groups of spaces [\cite{morel2012a1}, Chapter 6]. In doing so, we shall show that, over a perfect field $k$, the $\mathbb{P}^1$-homotopy groups are $\mathbb{P}^1$-invariant. However, our result is weaker in the sense that we neither prove nor disprove that the $i$-th $\mathbb{P}^1$-homotopy group is $i$-strong. Nor do we prove or disprove that strongly $\mathbb{P}^1$-invariant sheaves of abelian groups are automatically strictly $\mathbb{P}^1$-invariant.

As usual, we define the $i$-th $\mathbb{P}^1$-homotopy group by the following formula:
$$\pi_i^{\mathbb{P}^1}:=\pi_i^{nis}L_{pmot}.$$
\subsection{On the $\mathbb{P}^1$-path component $\pi_0^{\mathbb{P}^1}$ }\mbox{}

Let $\esc{X}\in \mathcal{H}^{\mathbb{P}^1}$ be a $\mathbb{P}^1$-motivic space. 

\begin{thm}\label{zar stalk in generic pi0}
    Let $X\in Sm_k$ be a connected scheme with generic point $\eta$. Then, for any $x\in X$, the map $\pi_0\esc{X}(X_x)\to \pi_0\esc{X}(\eta)$ has a trivial fiber. 
\end{thm}
\begin{proof}
The argument proceeds by translating Morel's $\mathbb{A}^1$-local Bloch-Ogus-Gabber strategy, specifically its rendition in [\cite{bachmann2024strongly}, Proposition 2.9], to the $\mathbb{P}^1$-local context. This translation is essentially adapted from Ayoub's argument in [\cite{ayoub2020P1}, Lemma 4.14].

We shall prove the following: For $X\in Sm_k$, a point $x\in X$, and a proper closed subset $Z'\subset X$ with $x\in Z'$ (the case $x \notin Z'$ being trivial), there exists an open neighborhood $X'$ of $x$ such that the fiber of $$\pi_0\esc{X}(X')\to \pi_0\esc{X}(X'\setminus Z')$$ is trivial. 

Let $X_1$ be an open neighborhood of $x\in X$ and $Z'\cap X_1=Z\subset X_1$ be a closed subscheme containing $x$ . By Gabber’s Lemma [\cite{MR1466971}] applied to $X_1, Z, x$, we obtain a smaller open neighborhood $X_2$ together with an étale neighborhood $X_2 \to \mathbb{A}^1_S$ of $Z\cap X_2$. Through the composite $Z\cap X_2\hookrightarrow X_2\to \mathbb{A}_S^1$ identify $Z\cap X_2$ as a closed subscheme of $\mathbb{A}^1_S$. Composing with the open embedding  $j_\infty: \mathbb{A}_S^1\hookrightarrow \mathbb{P}^1_S$ given as the complement of $\infty$, we find that $Z\cap X_2$ is also a closed subscheme of $\mathbb{P}^1_S$. Indeed, since  $Z \cap X_2 \to S$  is finite and $\mathbb{P}_S^1\to S$ is proper, the immersion $Z\cap X_2\hookrightarrow \mathbb{P}^1_S$ is proper and hence closed.  

\underline{\textit{Case 1: $k$ is infinite}.}

Since $Z \cap X_2 \to S$ is finite, when $k$ is infinite, the canonical map $\mathbb{P}^1_S \setminus Z \to S$ admits a section over an open neighborhood $S_1$ of the image of $x$ in $S$. To see this, let $s\in S$ be the image of $x$ and $\bar{s} \in S$ be a closed specialization of $s$; then $\mathbb{P}^1_{\bar{s}} \setminus Z_{\bar{s}}\to \bar{s}$ has a rational section because $k(\bar{s})$ is infinite. This defines a rational section $f: U\to \mathbb{P}^1_{\bar{s}}$ on a small neighborhood $U$ of $\bar{s}$. Now $f^{-1}(Z) \subset U$ is a closed subscheme that, by construction, does not contain $\bar{s}$. So $S_1=U\setminus f^{-1}(Z)$ works. Write $X_3$ for the preimage of $S_1$ in $X_2$.    
    
    We obtain a Nisnevich square as shown on the left:
\[
\xymatrix{
X_3\setminus Z \ar@{^(->}[r]\ar[d] & X_3\ar[d]&&\esc{X}(S)\ar[r]^{\simeq }& \esc{X}(\mathbb{P}^1_S)\ar[r]\ar[d]&\esc{X}(\mathbb{P}^1_S\setminus Z)\ar[d]\\
\mathbb{P}^1_S\setminus Z\ar@{^(->}[r]& \mathbb{P}^1_S&&& \esc{X}(X_3)\ar[r]&\esc{X}(X_3\setminus Z)
}
\]
Mapping it to the projective motivic (which is Nisnevich-local, therefore) space $\esc{X}$, we obtain the cartesian square on the right, where the top-left equivalence is due to $\mathbb{P}^1$-invariance of $\esc{X}$. Thus, the fiber of the given point yields an equivalence of pointed sections of pointed spaces $$\esc{X}(\mathbb{P}^1_{S_1}/\mathbb{P}^1_{S_1}\setminus Z)\simeq \esc{X}(X_3/X_3\setminus Z).$$ But since $\mathbb{P}^1_{S_1} \setminus Z \to  S $ is split we find that the morphism $$a:\pi_0\esc{X}(\mathbb{P}^1_{S_1})\to \pi_0\esc{X}(\mathbb{P}^1_{S_1}\setminus Z)$$ is a split monomorphism and hence injective. So that $$\pi_0(\esc{X}(\mathbb{P}^1_{S_1}/\mathbb{P}^1_{S_1}\setminus Z))\to \pi_0(\esc{X}(\mathbb{P}^1_{S_1}))$$ is the zero map whence so is $$\pi_0(\esc{X}(X_3/X_3\setminus Z))\to \pi_0(\esc{X}(X_3)).$$ So the fiber of $\pi_0\esc{X}(X_3)\to \pi_0\esc{X}(X_3\setminus Z)$ is indeed trivial. By taking the colimit over all $X_1$ and $Z$'s, we can conclude easily.

\underline{\textit{Case 2: $k$ is finite}.}

When $k$ is finite, we shall prove the claim by induction on the dimension of $X$. If $dimX=0$, connectivity implies $X=x$ and $Z=\emptyset$, so the claim holds vacuously.

When $dimX=1$, since there are finitely many closed points, we may remove all points except $x$ and assume $Z=x$. We may use the same proof as above. It remains to construct a section of $\mathbb{P}^1_S \setminus Z \to  S $. We may use the same arguments as before. Note, however, that the image of $x$ is already closed, so $\bar{s}=s$ is the image itself. As $s/k$ is a finite extension, $\mathbb{P}^1_s$ has at least three points. So $\mathbb{P}^1_s\setminus x\to s$ must have a section, and we are done.

For $dimX>1$, we choose $\eta_S$ instead of $\bar{s}$. Since $dimS=dimX-1>0$, the residue field $k(\eta_S)$ is infinite. As in the infinite case, we may find a section of $\mathbb{P}^1_S \setminus Z \to S$ over some open set $U\subset S$ (around $\eta_S$, not necessarily the image $s$). But then the inductive hypothesis applies to $(S,s, S\setminus U)$, and we can shrink $S$ to obtain $S_1$ around $s$ such that $\pi_0^{}\esc{X}(S_1)\to \pi_0^{}\esc{X}(S_1\cap U)$ has trivial fiber. We consider the following commutative diagram:
\[\xymatrix{
\pi_0\esc{X}{({S_1})} \ar[r]^{\simeq }\ar@{>->}[d]&\pi_0\esc{X}{(\mathbb{P}^1_{S_1})} \ar[r]^{a}\ar[d]& \pi_0\esc{X}{(\mathbb{P}^1_{S_1}\setminus Z)}\ar[d]\\
\pi_0\esc{X}{({S_1\cap U})} \ar[r]^{\simeq}&\pi_0\esc{X}{(\mathbb{P}^1_{S_1\cap U})} \ar[r]& \pi_0\esc{X}{(\mathbb{P}^1_{S_1\cap U}\setminus Z)}\\
}\]
Since the local section of $\mathbb{P}^1_S \setminus Z \to  S $ restricts to a global section of $\mathbb{P}^1_{S_1} \setminus Z \to  S_1 $, we obtain that the bottom horizontal map in the diagram above is split monic. Since the left vertical map has a trivial fiber, it follows that so does $a$.
\end{proof}
\begin{cor}\label[cor]{P1 connected on F_k}
 For a projective motivic space $\esc{X}$, the Nisnevich homotopy path component $\pi_0^{nis}\esc{X}$ is trivial if and only if $\pi_0^{nis}\esc{X}(K)=0$ for all $K\in \mathcal{F}_k$.
\end{cor}
\begin{proof}
    It suffices to show that  $\pi_0^{zar}\esc{X}$ is trivial. But from the theorem above for any point $x\in X\in Sm_k$ we see that $$\pi_0^{zar}\esc{X}(X_x)=\pi_0^{}\esc{X}(X_x)\hookrightarrow \pi_0^{zar}\esc{X}(\eta_X)=*$$ has a trivial fiber for any choice of base point. Hence   $\pi_0^{zar}\esc{X}$ is indeed trivial.
\end{proof}
\begin{cor}
    For any projective motivic $H$-group-space $\esc{X}$ and connected $X\in Sm_k$ with generic point $\eta$ the canonical map $\pi_0^{zar}\esc{X}(X)\to \pi_0^{zar}\esc{X}(\eta_X)$ is injective.
\end{cor}
\begin{proof}
    Since $$(\pi_0^{zar}\esc{X})(X_x)=\pi_0^{}\esc{X}(X_x)\text{ and }(\pi_0^{zar}\esc{X})(\eta_X)=\pi_0^{}\esc{X}(\eta_X),$$ applying the above theorem to $ \esc{X}$, pointed by the unit, and using the fact that $$\pi_1^{zar}\esc{X}(X_x)\to \pi_1^{zar}\esc{X}(\eta_X)$$ is a group homomorphism, we conclude that the last morphism is injective. Now we have a commutative triangle:
    \[
    \xymatrix{ 
     \pi_0^{zar}\esc{X}(X)\ar[r]\ar[dr]&\underset{x\in X}{\Pi}\pi_0^{zar}\esc{X}(X_x)\ar[d]\\
     & \pi_0^{zar}\esc{X}(\eta_X)}
    \]
    The top horizontal map is injective by the Zariski sheaf condition; hence, so is the composite.
\end{proof}
\begin{cor}
     For any projective motivic $H$-group-space $\esc{X}$ and connected $X\in Sm_k$ with generic point $\eta$, the canonical map $\pi_0^{nis}\esc{X}(X)\to \pi_0^{nis}\esc{X}(\eta_X)$ is injective.
\end{cor}

\begin{proof}
After restricting to $Et_X$, we may assume that all morphisms are smooth. We are thus dealing with the global sections of the Nisnevich sheafification of $(\pi_0^{zar}\esc{X})_{|Et_X}$ and the birational (pre)sheaf $\pi_0^{}\esc{X}(\eta_{-})_{|Et_X}$. Because the map is injective for Zariski sheaves and restriction and sheafification preserve injective maps, we are done.
\end{proof}
In the next theorem, we shall use the following definition:
\begin{defn}
    A presheaf $F$ of sets on $Sm_k$ is said to be weakly unramified, if for every dense open immersion $U\subset X$ in $Sm_S$ the induced restriction map $FX\to FU$ is injective.
\end{defn}
\begin{thm}\label{pi0 of grouplike is weak unram}
  For any projective motivic $H$-group-space $\esc{X}$, one has that $\pi_0^{nis}\esc{X}$ is weakly unramified. In particular, for any grouplike space $\esc{Y}$, the projective homotopy path component $\pi_0^{\mathbb{P}^1}\esc{Y}$ is weakly unramified. 
\end{thm}
\begin{proof}
    If $X\in Sm_S$ is connected, then for any dense open $U\subset X$, the induced morphism is injective because the diagonal in the diagram below is injective by the corollary just above:
    \[
    \xymatrix{
    \pi_0^{nis}\esc{X}(X)\ar[r]\ar[dr]& \pi_0^{nis}\esc{X}(U)\ar[d]\\
    &\pi_0^{nis}\esc{X}(\eta_X)
    }
    \]
    Since Nisnevich sheaves take disjoint unions to products, this works even for non-connected smooth schemes.

    For the last statement, note that since $L_{pmot}$ is a cartesian localization \Cref{pmot cartesian localization}, $L_{pmot}\esc{Y}$ is grouplike when $\esc{Y}$ is so.
\end{proof}

\underline{\textbf{$\mathbb{P}^1$-connected spaces}}
\begin{defn}
    For an $S$-space $\esc{X}$ we denote by $$\pi_{0}^{ch,\mathbb{P}^1}\esc{X}:=\pi_0^{nis}L_{\mathbb{P}^1}\esc{X}$$ and call it the $\mathbb{P}^1$-chain connected component sheaf of $\esc{X}$. Additionally, we define the directly $\mathbb{P}^1$-connected component of $\esc{X}$ by the co equalizer of the following diagram:
\[
    \xymatrix{
    \pi_0^{nis}\esc{X}^{\mathbb{P}^1}\ar@<-.5ex>[r]_-{0}\ar@<.5ex>[r]^-1&\pi_0^{nis}\esc{X}\ar[r]& \pi_0^{dir, \mathbb{P}^1}\esc{X}}\]
    induced by the two sections $0,1$.
\end{defn}
\begin{rem}\label[rem]{der a1 is ch a1}
    One could similarly define the notion of a directly $\mathbb{A}^1$-connected component. However, since $\mathbb{A}^1$ has an interval structure, this is unnecessary. In other words, $\pi_0^{der,\mathbb{A}^1}\to \pi_0^{ch,\mathbb{A}^1}$ is an isomorphism.
\end{rem}

\begin{defn}
    We say that an $S$-space $\esc{X}$ is:
    \begin{enumerate}
        \item  $\mathbb{P}^1$-connected if $\pi_0^{\mathbb{P}^1}\esc{X}$ is trivial.
        \item $\mathbb{P}^1$-chain connected if $\pi_{0}^{ch,\mathbb{P}^1}\esc{X}$ trivial.
        \item Weakly $\mathbb{P}^1$-chain connected if $\pi_{0}^{ch,\mathbb{P}^1}\esc{X}$ trivial on $\mathcal{F}_k$, i.e., for every finitely generated separable extension $L/k$ $$\pi_0(L_{\mathbb{P}^1}\esc{X}(L))=*.$$
    \end{enumerate}
\end{defn}
\begin{rem}
    By \Cref{model for pmot localization} there is a canonical epimorphism of sheaves $$\pi_0^{ch,\mathbb{P}^1}\esc{X}\to \pi_0^{\mathbb{P}^1}\esc{X}.$$ It follows that $\mathbb{P}^1$-chain connected spaces are $\mathbb{P}^1$-connected. In fact, when the base field $k$ is perfect, \Cref{P1 connected on F_k} yields that weakly $\mathbb{P}^1$-chain connected $k$-spaces are $\mathbb{P}^1$-connected.
\end{rem}
In the following result, we say that a variety $Y$ is directly rationally connected if every two rational points of $Y$ are connected by a $\mathbb{P}^1$. We say that $Y$ is entirely rationally connected if for every $L\in \mathcal{F}_k$ $Y_K$ is directly rationally connected, i.e. $\pi_0^{dir,\mathbb{P}^1}\esc{X}$ is trivial on $\mathcal{F}_k$.

\begin{prop}
    Entirely rationally connected varieties are weakly $\mathbb{P}^1$-chain connected. Thus, over a perfect field, such varieties are  $\mathbb{P}^1$-connected.
\end{prop}
\begin{proof}
   For the first statement, one has to show that for every finitely generated separable extension $L/k$ $$\pi_0(L_{\mathbb{P}^1}\esc{X}(L))=*.$$ Using the formula of $L_{\mathbb{P}^1}$ (\Cref{model for P1 localization}) and methods of the proof of \Cref{model for P1 localization} it is clear that the epimorphism $$\pi_0\esc{X}(L)\to \pi_0(L_{\mathbb{P}^1}\esc{X}(L))$$ is explicitely given by identifying points of $\esc{X}(L)$ throug sections from $(\mathbb{P}^1_L)^n$ to $\esc{X}$. For entirely rationally connected varieties, this happens already at the level $n=1$.
\end{proof}

\begin{lem}
    There is a span of comparison maps:
    $$\pi_0^{ch,\mathbb{A}^1}\esc{X}\leftarrow \pi_0^{dir,\mathbb{P}^1}\esc{X}\to \pi_0^{ch,\mathbb{P}^1}\esc{X}.$$
\end{lem}

\begin{proof}
    The leftward map constructed by choosing an embedding $j:\mathbb{A}^1\subset \mathbb{P}^1$, say as the complement of $\infty$. Indeed, by \Cref{der a1 is ch a1} the map is induced by the following map of diagrams:
    \[
    \xymatrix{
    \pi_0^{nis}\esc{X}^{\mathbb{P}^1}\ar@<-.5ex>[r]_-{0}\ar@<.5ex>[r]^-1\ar[d]_{j_\infty^*}&\pi_0^{nis}\esc{X}\ar@{=}[d]^{id}\\
    \pi_0^{nis} \esc{X}^{\mathbb{A}^1}\ar@<-.5ex>[r]_-0\ar@<.5ex>[r]^-1&\pi_0^{nis}\esc{X}
    }
    \]
    The rightward map, on the other hand, is obtained by the inclusion of the smaller subdiagram defining $\pi_0^{der,\mathbb{P}^1}\esc{X}$ into the larger diagram $ \pi_0^{nis}\esc{X}(\delta^\bullet_{\mathbb{P}_S^1}\times_S -)$ defining $\pi_0^{ch,\mathbb{P}^1}\esc{X}$.
\end{proof}
\begin{prop}
    For a proper variety $X/k$, the comparison maps from the lemma above are isomorphisms in one variable, yielding an isomorphism $$\pi_0^{ch,\mathbb{A}^1}X (L)\to \pi_0^{ch,\mathbb{P}^1}X(L)$$ for every $L\in \mathcal{F}_k.$ It follows that the canonical maps $$\pi_0^{ch, \mathbb{P}^1}X\to \pi_0^{\mathbb{P}^1}X\to \pi_0^{b\mathbb{A}^1}X$$ are isomorphisms on $\mathcal{F}_k$. Thus, there are equivalences     
    $$\pi_0^{\mathbb{A}^1}X (L)\simeq \pi_0^bX(L)\simeq   \pi_0^{\mathbb{P}^1}X(L).$$ 
\end{prop}
\begin{proof}
    The equivalence $$\pi_0^{ch,\mathbb{A}^1}{X}(L)\leftarrow \pi_0^{dir,\mathbb{P}^1}{X}(L)$$ follows from the uniqueness of the valuative criteria of proper schemes. Indeed, from this criterion, we have an isomorphism, $$\mathrm{Map}(\mathbb{P}^1_L,X)\xrightarrow{j_\infty}\mathrm{Map}(\mathbb{A}^1_L,X).$$ For the other equivalence, we use the fact that identifying points of ${X}(L)$ through sections from $(\mathbb{P}^1_L)^n$ to ${X}$ lifts to an identification by an $\mathbb{A}^1$ and thus through $\mathbb{P}^1$ by properness (here the key of the argument is that there is always an $\mathbb{A}^1$ passing through two distinct points in $(\mathbb{P}^1)^n$. Since $\pi_0^{ch, \mathbb{P}^1}X\to \pi_0^{\mathbb{P}^1}X$ is an epimorphism, the rest of the claim follows from [\cite{asok2011smooth}, Proposition 6.2.6]. While the last claim uses the isomorphism $\pi_0^bX\simeq \pi_0^{b\mathbb{A}^1}X$ from [\cite{0bat}, Theorem 3.6.15].
\end{proof}
It follows that,
\begin{cor}
A proper variety over $X$, over a perfect field $k$, is $\mathbb{P}^1$-connected iff it is birationally connected iff it is $\mathbb{A}^1$-connected.
\end{cor}
\begin{proof}
    Using \Cref{P1 connected on F_k} and [\cite{bachmann2024strongly}, Corollary 2.10] this follows from [\cite{0bat}, Corollary 3.6.22].
\end{proof}

\subsection{$\mathbb{P}^1$-invariance of $\pi_i^{\mathbb{P}^1}$}
\begin{defn}
    Over a field $k$, we say that a presheaf $F$ is weakly $\mathbb{P}^1$-invariant if, for all $K\in \mathcal{F}_k$, the canonical map $F(K)\to F(\mathbb{P}^1_K)$ is an isomorphism. 
\end{defn}
\begin{rem}\label[rem]{remark on weak p1 invariance}
\begin{enumerate}
    \item Choosing a section $K\xrightarrow[]{\infty}\mathbb{P}^1_K\to K$ shows that the canonical map $F(K)\to F(\mathbb{P}^1_K)$ is injective. Thus, weak $\mathbb{P}^1$-invariance is equivalent to the map $F(K)\to F(\mathbb{P}^1_K)$ being surjective. Conversely, this is equivalent to the map $F(\mathbb{P}^1_K)\to F(K)$ being injective.
    \item If $F$ is weakly $\mathbb{P}^1$-invariant and is a $\sqcup$-sheaf (i.e., it takes disjoint unions to products), then for any essentially smooth scheme $U$ of dimension $0$, the map $F(U)\to F(\mathbb{P}^1_U)$ is an isomorphism (for, after all, such a $U$ is a disjoint union of fields $K\in \mathcal{F}_k$). \end{enumerate}
\end{rem}

\begin{prop}\label[prop]{weakly p1 + weakly unramified => p1 invariant}
    A weakly unramified, weakly $\mathbb{P}^1$-invariant presheaf $G$ is $\mathbb{P}^1$-invariant. 
\end{prop}
\begin{proof}
 Let $U$ be a smooth, connected scheme with generic point $\eta=Spec K$. It suffices to show that the map $ G(\mathbb{P}^1_U)\to F(U)$ is injective. Consider the following commutative square:
    \[
    \xymatrix{ 
    G(\mathbb{P}^1_U)\ar[r]\ar[d]&G(\mathbb{P}^1_K)\ar[d]\\
    G(U)\ar[r]&G(K) }
    \]
    The right vertical map is an isomorphism by weak $\mathbb{P}^1$-invariance, whereas the top horizontal map is injective by weak unramifiedness. Indeed, $\mathbb{P}^1_K\subset \mathbb{P}^1_U$ is an essentially smooth dense open. We conclude that the left vertical map is injective, as required.
\end{proof}
\begin{thm}\label{pi 1 p1 is weak unram}
For all $i\geq 1$ and any space $\esc{X}$, the higher projective homotopy groups $\pi_i^{\mathbb{P}^1}\esc{X}$ are weakly unramified.
\end{thm}
\begin{proof}
Apply the above theorem \Cref{pi0 of grouplike is weak unram} to the grouplike space $\Omega^i L_{pmot}\esc{X}$.
\end{proof}
\begin{prop}\label[prop]{pi0 p1 is weakly p1 invariant}
Let $\esc{X}$ be a projective motivic space. Then $\pi_0^{nis}\esc{X}$ is weakly $\mathbb{P}^1$-invariant.
\end{prop}
\begin{proof}
By \Cref{remark on weak p1 invariance}, we have to show that for any $K\in \mathcal{F}_k$ the map $$\tau_{\leq 0}^{nis}\esc{X}(K)\to  \tau_{\leq 0}^{nis}\esc{X}(\mathbb{P}^1_K)$$ is surjective.  We have a commutative square of Nisnevich homotopy classes:
\[
\xymatrix{
[K,\esc{X}]\ar[r]\ar[d]^{\simeq }&[K,\tau_{\leq 0}^{nis}\esc{X}]\ar[d]\\
[\mathbb{P}^1_K, \esc{X}]\ar[r]&[\mathbb{P}^1_K, \tau_{\leq 0}^{nis}\esc{X}]
}
\]
    Both the top and bottom arrows are surjective by [\cite{bachmann2024strongly}, Lemma 2.17] and the fact that the homotopy dimension of both $\mathbb{P}^1_K$ and $K$ is $\leq 1$. The left vertical map is an isomorphism because of the $\mathbb{P}^1$-invariance of $\esc{X}$. 
\end{proof}
\begin{cor}\label[cor]{pi0 of grouplike is p1 invariant}
For a projective motivic $H$-group-space $\esc{G}$,  $\pi_0^{nis}\esc{G}$ is $\mathbb{P}^1$-invariant.
\end{cor}
\begin{proof}
Applying \Cref{pi0 p1 is weakly p1 invariant} to $\esc{G}$, we find that $\pi_0^{nis}\esc{G}$ is weakly $\mathbb{P}^1$-invariant. For an $H$-group-space $\esc{G}$, the weak unramifiedness of $\pi_0^{nis}\esc{G}$ is given by Proposition \Cref{pi0 of grouplike is weak unram}. The statement then follows at once from \Cref{weakly p1 + weakly unramified => p1 invariant}.
\end{proof}
\begin{cor}
Let $\esc{X}$ be a projective motivic space. Then for all $i\geq 1$, $\pi_i^{nis}\esc{X}$ is $\mathbb{P}^1$-invariant. 
\end{cor}
\begin{proof}
Apply \Cref{pi0 of grouplike is p1 invariant} above to the grouplike projective motivic space $\Omega^i \esc{X}$.
\end{proof}
\begin{cor}\label[cor]{p1 inv pf pi_i^p1}
For any $i\geq 1$ and any space $\esc{X}$, the projective homotopy groups $\pi_i^{\mathbb{P}^1}\esc{X}$ are $\mathbb{P}^1$-invariant.
\end{cor}
\begin{proof}
Apply the above corollary to the projective motivic space $L_{pmot}\esc{X}$.
\end{proof}
\subsection{Strong/strict $\mathbb{P}^1$-invariance}
\begin{defn}
     A Nisnevich sheaf $G$ of groups is said to be 
    \begin{enumerate}
        \item $i$-strictly $\mathbb{P}^1$-invariant if $\mathrm{H}_{nis}^{j}(-;G)$ is $\mathbb{P}^1$-invariant for all $j\leq i$ (for $i\geq 2$, this requires $G$ to be abelian).
        \item strongly $\mathbb{P}^1$-invariant if it is $1$-strictly $\mathbb{P}^1$-invariant. In detail, this means that ${H}^1_{nis}(-;G)$ and $G$ are $\mathbb{P}^1$-invariant. 
        \item strictly $\mathbb{P}^1$-invariant if it is $\infty$-strictly $\mathbb{P}^1$-invariant. This means that $G$ is abelian and that $\mathrm{H}_{nis}^{i}(-;G)$ is $\mathbb{P}^1$-invariant for all $i$. 
    \end{enumerate}
\end{defn}
\begin{rem}
$G$ is $i$-strictly $\mathbb{P}^1$-invariant if and only if $\mathrm{B}_{nis}^jG$ is $\mathbb{P}^1$-local for all $j\leq i$. In particular, when $i=1$, this is precisely the discrete case of \Cref{strong p1 monoid space}.
\end{rem}
\begin{ex}[$\mathbb{A}^1$ is strictly $\mathbb{P}^1$ invariant]
    Since all higher Zariski cohomologies of $\mathbb{P}^1_S$ with respect to the coherent sheaf $\mathcal{O}_{\mathbb{P}^1}=\mathcal{O}_{\mathbb{P}^1}(0)$ are trivial [\cite[\href{https://stacks.math.columbia.edu/tag/01XX}{Lemma 01XX}]{stacks-project}], we may conclude the claim Zariski locally by the Leray spectral sequence. But we know that for quasicoherent sheaves, Zariski and Nisnevich cohomologies agree [\cite[\href{https://stacks.math.columbia.edu/tag/03P2}{Theorem 03P2}]{stacks-project}], so the claim follows.
\end{ex}

Just as in the $\mathbb{A}^1$ case, one may wish that a $\mathbb{P}^1$-invariant sheaf of groups be automatically strongly $\mathbb{P}^1$-invariant. The $\mathbb{A}^1$-homotopical counterpart was disproved by [\cite{choudhury2014connectivity}]. Unfortunately, the $\mathbb{P}^1$-case is even worse, in the sense that the counterexample is simpler. 
\vspace{.3 cm}
\begin{coex}[$\mathbb{G}_m$ is $\mathbb{P}^1$-invariant but not strongly $\mathbb{P}^1$-invariant]
Indeed, we have seen that $\mathbb{G}_m$ is an affine scheme [\Cref{affine algebras are projective local}] and is $\mathbb{P}^1$-local. However, the counterexample \Cref{Gm is not strong p1 local} shows that $\mathbb{G}_m$ is not strongly $\mathbb{P}^1$-invariant.
\end{coex}

Thus, in the $\mathbb{P}^1$ case as well, it is not sufficient to have the $\mathbb{P}^1$-invariance of higher $\mathbb{P}^1$-homotopy groups to obtain their strong/strict $\mathbb{P}^1$-invariance. 

In the case of $\mathbb{A}^1$ localization, we know that, at least over perfect fields, all the $i$-th homotopy group sheaves are $i$-strictly $\mathbb{A}^1$-invariant. In the $\mathbb{P}^1$-local case, we have so far only been able to deduce $\mathbb{P}^1$-invariance, not the strong $\mathbb{P}^1$-invariance. Most of Morel's $\mathbb{A}^1$-local computations showing 'strongness' can be carried forward to our setup as well; the only missing part is the purity theorem, and at this point I am unable to even state an appropriate version of the $\mathbb{P}^1$-local purity theorem. In the next subsection, we shall discuss the $\mathbb{P}^1$-local version of another key result of Morel that relies on purity.
\subsection{Rational contraction}
Let $\mathbb{A}$ denote the pointed sheaf $(\mathbb{A}^1, 0)$. Let $X\subset \mathbb{P}^1_k$ be an open subscheme with structure map $p_X: X\to SpecK$ and a chosen rational point $x: Spec k\to X$. We shall denote by $\mathbb{X}: =(X,x)$ the space $h_k(X)$ pointed by $x$. 
\begin{defn}
\begin{enumerate}
    \item     Given a pointed space $\esc{A}$, the $\mathbb{X}$-contraction is given by $$\esc{A}_{-\mathbb{X}}:=\esc{A}^{\mathbb{A}}=\Omega_\mathbb{X}\esc{A}.$$ 
    \item When $A$ is a sheaf of pointed sets, $A_{-\mathbb{X}}$ can be identified with the sheaf $$U\mapsto ker(A(\mathbb{X}_U)\xrightarrow[]{0}A(U))$$
We shall denote this by $A^{-1}$ when $\mathbb{X}=\mathbb{A}$.
\end{enumerate}
\end{defn}
\begin{obs}
    \begin{enumerate}
        \item If $\esc{X}$ is grouplike, then so is $\esc{X}_{-\mathbb{X}}$. In particular, when $F$ is a sheaf of (abelian) groups, then so is $F^{-1}$.
        \item When $\esc{X}$ is local with respect to a class of morphisms, then so is $\esc{X}_{-\mathbb{X}}$. In particular, when $\esc{X}$ is an affine or projective motivic space, so is $\esc{X}$.
        \item If $F$ is a sheaf of pointed sets, then $F^{-1}$ is a retract of the unpointed mapping space $F^{\mathbb{A}^1}$. When $F$ is a sheaf of abelian groups, this implies that $F^{-1}$ is a summand of $F(\mathbb{X}\times -)$.
    \end{enumerate}
\end{obs}
$\mathbb{A}^1$-locally we know that $\Omega_{\mathbb{G}}$ commutes with $\pi_i^{nis}$. Unfortunately, $\mathbb{P}^1$-locally without an appropriate purity theorem, at this point, we are able to deduce this only for Eilenberg-MacLane spaces:
\begin{prop}\label[prop]{x contraction of eilenberg}
Let $n$ be a natural number. Let $A$ be an $n$-strictly $\mathbb{P}^1$-invariant sheaf of $n$-groups. Then we have an equivalence $$\Omega_\mathbb{X}K^{nis}(A,n)\simeq K^{nis}(A_{-\mathbb{X}}, n).$$ 
\end{prop}
\begin{proof}
    For $n=0$, this is just the definition. 
    First of all, being a right adjoint, $\Omega_\mathbb{X}$ preserves $n$-truncated objects [\cite{lurie2009higher}, Proposition 5.5.6.16]. It follows that $\Omega_\mathbb{X}K^{nis}(A,n)$ is $n$-truncated. Clearly, $$\Omega^n\Omega
    _\mathbb{X}\mathrm{K}^{nis}(A,n)\simeq \Omega_\mathbb{X}\Omega^n\mathrm{K}^{nis}(A,n)\simeq \Omega_\mathbb{X}A=A_{-\mathbb{X}}.$$ Since for $0< i\leq n$ $$\Omega^{n-i}\Omega_\mathbb{X}\mathrm{K}^{nis}(A,n) \simeq \Omega_\mathbb{X}\mathrm{K}^{nis}(A,i),$$ it suffices to show that $\Omega_\mathbb{X}\mathrm{K}^{nis}(A,i)$ is connected. Observe that since $i\leq  n$ and $A$ is $n$-strictly $\mathbb{P}^1$-invariant, it is $i$-strictly $\mathbb{P}^1$-invariant as well. That is $\Omega_\mathbb{X}\mathrm{K}^{nis}(A,i)$ is $\mathbb{P}^1$-local. So by \Cref{P1 connected on F_k} it reduces to show that for any $K\in \mathcal{F}_k$ $$\pi_0{(\Omega_\mathbb{X}\mathrm{K}^{nis}(A,i))}(K)=[\mathbb{X}_K, \mathrm{K}^{nis}(A,i)]_\bullet=0.$$ But this is a retract of $$[\mathbb{X}_K, \mathrm{K}^{nis}(A,i)]=H^i_{nis}(X_K, A)$$ by  the projection map $X\to Speck$. Since the Nisnevich cohomological dimension of $X_K$ is $1$, we only need to consider $i=1$. By the strong $\mathbb{P}^1$-invariance of $A$ and the fact that $X$ has Krull dimension $1$, we know that $$H^1_{nis}(X_K, A)\cong H^1_{zar}(X_K,A)$$ by the lemma below. Now, a Zariski cover of $X_K$ can easily be extended to $\mathbb{P}^1$ by adding the leftover point to one of the open sets. By renaming the cocycles, one thus finds a surjective map $$H^1_{zar}(\mathbb{P}^1_K,A)\to H^1_{zar}(X_K,A).$$ But due to strong $\mathbb{P}^1$-invariance of $A$ and the $0$ cohomological dimensionality of fields, one obtains $$H^1_{zar}(\mathbb{P}^1_K,A)\cong H^1_{zar}(K,A)=*,$$ rendering $H^1_{zar}(X_K,A)=*$. 
\end{proof}
\begin{lem}[2.11]
    Let $A$ be an $n$-strictly $\mathbb{P}^1$-invariant and $X \in \mathrm{Sm}_k$. Then $H^i_{\mathrm{Zar}}(X, A) = H^i_{\mathrm{Nis}}(X, A)$ for $i \leq m$.
\end{lem}

\begin{proof}
   We shall show that $H^i_{\mathrm{Nis}}(X_x, F) = 0$ for all $x \in X$ and $i \neq 0$. Now, by applying \Cref{zar stalk in generic pi0} to the projective motivic space $\mathrm{K}(A, i)$, we know that the canonical map $$H^i(X_x, A) \to H^i(\eta, A)$$ has trivial fiber, where $\eta$ is the generic point of the component of $x$. But since fields have Nisnevich cohomological dimension $0$, $H^i(\eta, A)=*$, and we are done.
\end{proof}
\begin{cor}
Let $n$ be a natural number. Let $A$ be a $n$-strictly $\mathbb{P}^1$-invariant sheaf of $n$-groups. Then so is $A_{-\mathbb{X}}$.
\end{cor}
\begin{proof}
    This follows at once from \Cref{x contraction of eilenberg} and that $n$-strict $\mathbb{P}^1$-invariance is equivalent to the $\mathbb{P}^1$-locality of $\mathrm{K}^{nis}(-, n)$.
\end{proof}
\begin{rem}
    Given a 1-dimensional rational subscheme $X$ of $P$ pointed by a rational point $x\in X$ to machieve an equivalence $$(\pi_i^{P})_{-\mathbb{X}}=\Omega_{\mathbb{X}}\pi_i^{P}\simeq \pi_i^P\Omega_{\mathbb{X}}$$ one requires a version of $P$-local purity. In the $P=\mathbb{A}^1$ case this is true for $\mathbb{X}=\mathbb{G}_m$ [\cite{morel2012a1}, Theorem 6.13]. In the $\mathbb{P}^1$ local case, we expect this for $\mathbb{X}=\mathbb{A}^1$, but to the best of my knowledge, this would require the `strong'ness of the higher homotopy group sheaves, and thus purity in disguise. 
\end{rem}

\section{Stable \texorpdfstring{ $\mathbb{P}^1$}{{P}1}-homotopy} 

\subsection{Model for stable $\mathbb{P}^1$-spectras}\mbox{}

Let $\mathcal{SH}^{\mathbb{P}^1}_{S^1}(S)$ be the stabilization of $\mathcal{H}^{\mathbb{P}^1}(S)$. By formal arguments, this is the localization of  $\mathcal{SH}^{nis}_{S^1}(S)$ at the set $$\{\Sigma^\infty \mathbb{P}^1_{X+}\to \Sigma^\infty X_+\}_{X\in Sm_S}.$$

In this section, we want to establish a connectivity theorem in the line of [\cite{morel2005stable}]. The idea is just a simple translation of Morel's methods as adapted by [\cite{ayoub2020P1}] for the abelian derived case. We claim no originality for these results. The main result is the following:
\begin{thm}
    Let $k$ be a field. The stable $\mathbb{P}^1$-localization functor $$L_{pmot}^{sp}: \mathcal{SH}^{nis}_{S^1}(k)\to \mathcal{SH}^{nis}_{S^1}(k)$$ preserves connectivity. 
\end{thm}
We will prove the above theorem in two steps. But first, let us restate how to construct a model for $\mathbb{P}^1$-localization for spectra. For the rest of the subsection, we shall suppress the base by assuming that we are working with a fixed qcqs base scheme $S$.
\begin{con}
Let $\mathbb{P}$ be the infinite suspension of the pointed projective line $(\mathbb{P}^1,0)$. Define $\Phi_\mathbb{P}$ as the cofiber of the evaluation $$\mathbb{P}\otimes (-)^{\mathbb{P}}\to 1_{\mathcal{SH}(S)}\to \Phi_\mathbb{P}.$$ By construction, it thus comes with a canonical map $1_{\mathcal{SH}(S)}\to \Phi_\mathbb{P}$. And let $$\Phi^\infty_\mathbb{P}:=\dcolim_{n}\Phi^n_\mathbb{P}$$ be the limit of the tower $$1_{\mathcal{SH}(S)}\to \Phi_\mathbb{P}\to \Phi^2_\mathbb{P}\to \cdots.$$
\end{con}
\begin{prop}\label[prop]{model for stable p1 bousfield localization}
    $\Phi^\infty_\mathbb{P}$ is a model for the monoidal localization of $\mathcal{SH}_{S^1}$ at the class $$\{\Sigma^\infty \mathbb{P}^1_{X+}\to \Sigma^\infty X_+\}_{X\in Sm_S}.$$
\end{prop}
\begin{proof}
    Apply [\cite{nStbat}, Proposition 2.2.2].
\end{proof}
\begin{thm}\label{formua for p1 localization spectra}
    Over an arbitrary base, there is an equivalence of localization functors $$L_{pmot}^{sp}\simeq L^{sp}_{\mathbb{P}^1,nis}\simeq \dcolim_{n}\big(\Phi_{\mathbb{P}^1} L_{nis}^{sp}\big )^n.$$ 
\end{thm}
\begin{proof}
The composit $$1\to L_{nis}^{sp}\to \dcolim_{n}\big(\Phi_{\mathbb{P}^1} L_{nis}^{sp}\big )^n$$ is a projective motivic-equivalence with a $\mathbb{P}^1$-local target. It is thus enough to prove that for every $\mal{X}\in \mathcal{SH}_{S^1}$, the object $$\dcolim_{n}\big(\Phi_{\mathbb{P}^1} L_{nis}^{sp}\big )^n\mal{X}$$ is nisnevich local. This is again clear from the cofinality argument and the filtered colimit stability of Nisnevich local spectra (similar to \Cref{model for P1 localization}).  
\end{proof}

\subsection{Stable $\mathbb{P}^1$ connectivity}\mbox{}

In this subsection, let the base be the spectrum of a field. The following definition is the spectral analog of those in [\cite{ayoub2020P1}].
\begin{defn}
\begin{enumerate}
\item  A presheaf of spectra $\mal{A}$ is said to be $n$-preconnected if, for all essentially smooth schemes $\mathcal{O}/k$, the spectra $\mal{A}(\mathcal{O})$ is $n - \dim\mathcal{O}$-connected.
\item    A presheaf of spectra $\mal{A}$ is said to be generically or weakly $n$-connected if, for all $K\in \mathcal{F}_k$, the spectrum $\mal{A}(K)$ is $n$-connected.
\end{enumerate}
It is immediate from the definition that a pre connected spectra is weakly connected.
\end{defn}
\begin{lem}[Weak $n$-connectivity]\label[lem]{Lpmot takes stable connected spectra to weakly connected spectra}
Over any field, the projective motivic localization $L_{pmot}^{sp}$ takes connected spectra to weakly connected spectra.
\end{lem}
\begin{proof}
 Since $\mathbb{P}^1$ is a one-dimensional scheme with a rational point, it follows from [\cite{nStbat}, Lemma 2.3.3] that $L_{pmot}^{sp}$ takes nisnevich-connected spectra to pre-connected spectra, which are then weakly connected, as observed before.
\end{proof}
\begin{cor}\label[cor]{weakly connected to connected}
    Over a field, a weakly connected $\mathbb{P}^1$-motivic spectrum is Nisnevich connected. 
\end{cor}
\begin{proof}
    The same proof as in \Cref{zar stalk in generic pi0} can be copied. In the third section of the proof of referred result, we apply the Nisnevich square to the presheaf $\pi_0^{sp}\mal{A}$ to obtain the injective map $$\pi_0^{sp}\mal{A}(X_x)\to \pi_0^{sp}\mal{A}(\eta)$$ (injectivity follows because all the maps are now group homomorphisms). By weak connectivity, we find that for any local ring $X_x$, $\pi_0^{sp}\mal{A}(X_x)=0$. This implies that $\pi_0^{sp,nis}\mal{A}=0$.
\end{proof}
\begin{cor}\label[cor]{projective stable connectivity}
    Over any field $k$, the projective motivic localization functor $L_{pmot}^{sp}$ for $S^1$-spectra preserves nisnevich $n$-connectivity. 
\end{cor}
\begin{proof}
    Since $L_{pmot}^{sp}$ is an exact functor, it suffices to show that it preserves connected spectra. By \Cref{Lpmot takes stable connected spectra to weakly connected spectra}, $L_{pmot}^{sp}$ takes connected spectra to weakly connected $\mathbb{P}^1$-motivic spectra, which are Nisnevich connected by \Cref{weakly connected to connected}.
\end{proof}
\begin{prop}\label[prop]{t structure for P1 motivic spectras}
The pair of full subcategories $$(\mathcal{SH}^{\mathbb{P}^1}_{S^1}\bigcap \mathcal{SH}_{S^1,\geq 0},\mathcal{SH}^{\mathbb{P}^1}_{S^1}\bigcap \mathcal{SH}_{S^1,\leq 0})$$ is an accessible complete $t$-structure on $\mathcal{SH}^{\mathbb{P}^1}_{S^1}$. The heart of this $t$-structure is identical to the category of strictly $\mathbb{P}^1$-invariant sheaves of abelian groups.
\end{prop}
\begin{proof}
    This follows from [\cite{nStbat}, Proposition 4.1.1], using Remark 4.1.2 in loc. cit. and the stable connectivity of $\mathbb{P}^1$ \Cref{projective stable connectivity}. For right completeness, observe that because Nisnevich excision is stable under filtered colimits, $\mathcal{SH}^{\mathbb{P}^1}_{S^1}(k)\subset\mathcal{SH}_{S^1}(k)$ is closed under all filtered colimits. The identification of the heart follows from the same Proposition of \textit{loc.cit.} referred to above.
\end{proof}
\begin{thm}\label{LAYERS OF P1 SPECTRA}
The following are equivalent for a nisnevich sheaf $\mal{A}$ of spectras on $k$
    \begin{enumerate}
        \item $\mal{A}$ is $\mathbb{P}^1$-local.
        \item All Nisnevich connective covers $\mal{A}_{\geq n}$ are $\mathbb{P}^1$-local.
        \item All nisnevich sheaves of homotopy groups $\pi_i^{sp,nis}\mal{A}$ are strictly $\mathbb{P}^1$-invariant.
    \end{enumerate}
  \end{thm}
  
\begin{proof}
    In the presence of \Cref{t structure for P1 motivic spectras}, this follows at once from the methods of [\cite{nStbat}, Theorem 4.1.8]. 
\end{proof}

\noindent\rule{\textwidth}{0.4pt}\\[1ex] 
\noindent{\large\textbf{\textit{Acknowledgments}:}} 

I thank Dr. Utsav for mentioning Ayoub's article on $\mathbb{P}^1$-localization [\cite{ayoub2020P1}] to me during a conversation.

\phantomsection
\bibliographystyle{amsalpha}	
\renewcommand\refname{Bibliography}
\bibliography{references}
\noindent\rule{\textwidth}{0.4pt} 
\end{document}